\documentclass[12pt]{amsart}

\usepackage{mathrsfs}
\usepackage{xcolor}
\usepackage{amsmath}
\usepackage{amsfonts}
\usepackage{amssymb}
\usepackage{amsthm}
\usepackage{amsopn}
\usepackage{amscd}
\usepackage{pstricks}
\usepackage[utf8]{inputenc}
\usepackage{courier}
\usepackage{url}
\usepackage{color}
\usepackage{subfigure}
\usepackage{xypic}
\usepackage{graphicx}
\usepackage{paralist}
\usepackage[all,cmtip]{xy}
\usepackage{amsbsy} 
\usepackage{hyperref} 
\usepackage{booktabs}

\usepackage{mathtools}

\DeclarePairedDelimiter{\floor}{\lfloor}{\rfloor}

\usepackage[vcentering,dvips]{geometry}
\entrymodifiers={+!!<0pt,\fontdimen22\textfont2>}
\title{Equivariant Vector Bundles on $T$-Varieties}

\author[N.~Ilten]{Nathan Ilten}
\address{Department of Mathematics,
        Simon Fraser University,
        8888 University Drive,
      Burnaby, BC V5A 1S6,
    Canada}
\email{nilten@sfu.ca}
\author[H.~S\"u\ss{}]{Hendrik S\"u\ss{}}
\address{
Hendrik S\"u\ss\\ School of Mathematics
The University of Edinburgh
James Clerk Maxwell Building
The King's Buildings
Mayfield Road
Edinburgh EH9 3JZ
}
\email{suess@sdf-eu.org}

\definecolor{lcolor}{rgb}{1.0,0.3,0.0}

\newcommand{\CC}{\mathbb{C}}
\newcommand{\QQ}{\mathbb{Q}}
\newcommand{\Q}{\mathcal{Q}}

\newcommand{\ZZ}{\mathbb Z}
\newcommand{\NN}{\mathbb N}

\renewcommand{\AA}{\mathbb{A}}
\newcommand{\PP}{\mathbb{P}}
\newcommand{\GG}{\mathbb{G}}
\newcommand{\tv}{\mathbf{X}}
\newcommand{\tV}{\widetilde{\V}}
\newcommand{\tE}{\widetilde{\E}}
\newcommand{\bF}{\mathbf{F}}

\newcommand{\D}{\mathcal D}
\newcommand{\E}{\mathcal E}
\newcommand{\G}{\mathcal G}
\newcommand{\A}{\mathbb A}
\newcommand{\C}{\mathcal C}

\newcommand{\V}{\mathcal V}
\newcommand{\VV}{\mathbb V}
\newcommand{\W}{\mathcal W}

\newcommand{\mcP}{\mathcal P}
\newcommand{\CO}{\mathcal O}
\newcommand{\mcL}{\mathcal L}
\newcommand{\F}{\mathcal F}
\newcommand{\T}{\mathcal T}

\newcommand{\xray}{\mathcal{H}}

\newcommand{\co}{\mathfrak{o}}

\newcommand{\bv}{\mathbf{v}}

\newcommand{\sch}{K\textrm{-}\mathit{Sch}}
\newcommand{\schx}{(\mathit{Sch}/X)}

\DeclareMathOperator{\Pic}{Pic}
\DeclareMathOperator{\Sym}{Sym}
\DeclareMathOperator{\Spec}{Spec}
\DeclareMathOperator{\Proj}{Proj}
\DeclareMathOperator{\rk}{rk}
\DeclareMathOperator{\WDiv}{WDiv}

\DeclareMathOperator{\spec}{Spec}

\DeclareMathOperator{\Ext}{Ext}

\DeclareMathOperator{\id}{id}

\DeclareMathOperator{\Hom}{Hom}
\DeclareMathOperator{\aut}{Aut}
\DeclareMathOperator{\Bl}{Bl}
\DeclareMathOperator{\Char}{char}

\newtheorem{theorem}{Theorem}[section]
\newtheorem*{theorem*}{Theorem}

\newtheorem{question}{Question}[theorem]

\newtheorem{prop}[theorem]{Proposition}
\newtheorem{lemma}[theorem]{Lemma}
\theoremstyle{definition}
\newtheorem{definition}[theorem]{Definition}
\newtheorem{rem}[theorem]{Remark}
\newtheorem{construction}[theorem]{Construction}
\newtheorem{ex}[theorem]{Example}
\newcommand{\fana}{%
 \psset{unit=1cm}
 \begin{pspicture}(-1.8,-2)(1.8,2)%

 \psset{linewidth=1pt}%
 \psline{->}(1,0)
 \psline{->}(-1,0)
 \psline{->}(0,1)
 \psline{->}(0,-1)
   \psgrid[gridwidth=0.3pt,griddots=5,subgriddiv=1,gridlabels=5pt](-1,-1)(1,1)

\end{pspicture}}

\newcommand{\fanb}{%
 \psset{unit=1cm}
 \begin{pspicture}(-1.8,-2)(1.8,2)%

 \psset{linewidth=1pt}%
 \psline{->}(1,0)
 \psline{->}(-1,-1)
 \psline{->}(0,1)
 \psline{->}(1,1)
   \psgrid[gridwidth=0.3pt,griddots=5,subgriddiv=1,gridlabels=5pt](-1,-1)(1,1)
\end{pspicture}}

\newcommand{\prefanb}{%
 \psset{unit=1cm}
 \begin{pspicture}(-1.8,-2)(1.8,2)%
 \psset{linewidth=1pt}%
 \rput(-0.23,0.5){$1$}
 \rput(-0.3,-0.5){$-1$}
 \rput(0.3,0.5){$1'$}
 \psline{->}(.06,0)(.06,1)
 \psline{->}(-.06,0)(-.06,1)
 \psdot(0,0)
 \psline{->}(0,0)(0,-1)
\end{pspicture}}

\newcommand{\systemfanb}{%
 \psset{unit=1cm}
 \begin{pspicture}(-1.8,-2)(1.8,2)%
 \psset{linewidth=1pt}%
 
 \psline{->}(0,0)(0,1)
 \psdot(0,0)
 \rput(-0.2,0.6){$1$}
 \rput(-0.3,0){$0$}
 \rput(-0.3,-0.6){$-1$}
 \psline{->}(0,0)(0,-1)

  \rput(0.8,0.6){$1'$}
\rput(0.8,0){$0$}
\psline{->}(0.5,0)(0.5,1)
 \psdot(0.5,0)
 \psline{->}(0.5,0)(0.5,-1)
 \rput(0.8,-0.6){$-1$}

\end{pspicture}}

\newcommand{\fanbproject}{%
 \psset{unit=1cm}
 \begin{pspicture}(-5,-2.5)(1.8,2)%
 \psset{linewidth=1pt}%
 \psline{->}(-4.06,0)(-4.06,1)
 \psline{->}(-3.93,0)(-3.93,1)
 \psline{->}(-2.5,-0.5)(-3.5,-0.5)
 \rput(-3,-0.2){$\mu$}
 \psdot(0,0)
 \psdot(-4,0)
\psline{->}(-4,0)(-4,-2)
 \psline[linestyle=dashed]{->}(1,0)
 \psline{->}(-2,-2)
 \psline{->}(0,1)
 \psline{->}(1,1)
   \psgrid[gridwidth=0.3pt,griddots=5,subgriddiv=1,gridlabels=5pt](-2,-2)(1,1)
   \psgrid[gridlabelcolor=white,gridwidth=0.3pt,griddots=5,subgriddiv=1,gridlabels=5pt](-5,-2)(-5,-2)(-4,1)
 \psdot[linecolor=black,linewidth=0.7pt,dotstyle=o](-4.06,1)
 \psdot[linecolor=black,linewidth=0.7pt,dotstyle=o](-3.93,1)
 \psdot[linecolor=black,linewidth=0.7pt,dotstyle=o](-4,-2)
\ \psdot[linecolor=black,linewidth=0.7pt,dotstyle=o](-4,-1)
 
\end{pspicture}}

\newcommand{\fanfproject}{%
 \psset{unit=1cm}
 \begin{pspicture}(-4,-2.5)(1.8,2)%
 \psset{linewidth=1pt}%
 \psline{->}(-3.06,0)(-3.06,1)
 \psline{->}(-2.93,0)(-2.93,1)
 \psline{->}(-1.5,-0.5)(-2.5,-0.5)
 \rput(-2,-0.2){$\mu$}
 \psdot(-3,0)
 \psdot(0,0)
 \psline{->}(-3,0)(-3,-2)
 \psline{->}(-1,-2)
 \psline{->}(0,1)
 \psline{->}(1,1)
   \psgrid[gridwidth=0.3pt,griddots=5,subgriddiv=1,gridlabels=5pt](-1,-2)(1,1)
   \psgrid[gridlabelcolor=white,gridwidth=0.3pt,griddots=5,subgriddiv=1,gridlabels=5pt](-4,-2)(-4,-2)(-3,1)
 \psdot[linecolor=black,linewidth=0.7pt,dotstyle=o](-3.06,1)
 \psdot[linecolor=black,linewidth=0.7pt,dotstyle=o](-2.93,1)
 \psdot[linecolor=black,linewidth=0.7pt,dotstyle=o](-3,-2)
\end{pspicture}}

\newcommand{\fanc}{%
 \psset{unit=1cm}
 \begin{pspicture}(-1.8,-2)(1.8,2)%
 \psset{linewidth=1pt}%
 \psline{->}(1,0)
 \psline{->}(-1,-1)
 \psline{->}(0,1)
 \psline{->}(-1,0)
 \psline{->}(0,-1)
   \psgrid[gridwidth=0.3pt,griddots=5,subgriddiv=1,gridlabels=5pt](-1,-1)(1,1)
\end{pspicture}}

\newcommand{\fand}{%
 \psset{unit=1cm}
 \begin{pspicture}(-1.8,-2)(1.8,2)%

 \psset{linewidth=1pt}%
 \psline{->}(1,0)
 \psline{->}(-1,-1)
 \psline{->}(0,1)
 \psline{->}(1,1)
 \psline{->}(-1,0)
 \psline{->}(0,-1)

   \psgrid[gridwidth=0.3pt,griddots=5,subgriddiv=1,gridlabels=5pt](-1,-1)(1,1)
\end{pspicture}}

\newcommand{\dega}{%
 \psset{unit=1cm}
 \begin{pspicture}(-1.8,-2)(1.8,2)%

 \psset{linewidth=1pt}%
   \psgrid[gridwidth=0.3pt,griddots=5,subgriddiv=1,gridlabels=5pt](-1,-1)(1,1)
\psdots[dotstyle=o,dotsize=.3](0,0)
\psdots (1,0)(0,1)(0,0)(-1,0)(0,-1)

\end{pspicture}}

\newcommand{\degb}{%
 \psset{unit=1cm}
 \begin{pspicture}(-1.8,-2)(1.8,2)%

 \psset{linewidth=1pt}%
   \psgrid[gridwidth=0.3pt,griddots=5,subgriddiv=1,gridlabels=5pt](-1,-1)(1,1)
   \psdots (0,0)
(1,0)
(0,1)
(1,-1)
(-1,1)

\end{pspicture}}

\newcommand{\degc}{%
 \psset{unit=1cm}
 \begin{pspicture}(-1.8,-2)(1.8,2)%

 \psset{linewidth=1pt}%
   \psgrid[gridwidth=0.3pt,griddots=5,subgriddiv=1,gridlabels=5pt](-1,-1)(1,1)
\psdots[dotstyle=o,dotsize=.3](0,0)(-1,0)(0,-1)
   \psdots(1,0)
(0,1)
(1,-1)
(-1,1)
(0,-1) 
(-1,0) 
(0,0)

\end{pspicture}}
\newcommand{\degd}{%
 \psset{unit=1cm}
 \begin{pspicture}(-1.8,-2)(1.8,2)%

 \psset{linewidth=1pt}%
   \psgrid[gridwidth=0.3pt,griddots=5,subgriddiv=1,gridlabels=5pt](-1,-1)(1,1)
   \psdots[dotstyle=o,dotsize=.5](0,0)

   \psdots[dotstyle=o,dotsize=.3](1,0)
(0,1)
(1,-1)
(-1,1)
(0,-1) 
(-1,0) 
(0,0)

\psdots(1,0)
(0,1)
(1,-1)
(-1,1)
(0,-1) 
(-1,0) 
(0,0)
\end{pspicture}}

\begin{document}
\begin{abstract}
Let $X$ be a $T$-variety, where $T$ is an algebraic torus. We describe a fully faithful functor from the category of $T$-equivariant vector bundles on $X$ to a certain category of filtered vector bundles on a suitable quotient of $X$ by $T$. We show that if $X$ is factorial, this functor gives an equivalence of categories. This generalizes Klyachko's description of equivariant vector bundles on toric varieties. We apply our machinery to show that vector bundles of low rank on $\PP^n$ which are equivariant with respect to special subtori of the maximal torus must split, generalizing a theorem of Kaneyama. Further applications include descriptions of global vector fields on $T$-varieties, and a study of equivariant deformations of equivariant vector bundles. 
\end{abstract}
\maketitle
\section{Introduction}
In \cite{klyachko:89a}, Klyachko described torus-equivariant vector bundles on toric varieties in terms of collections of filtered vector spaces satisfying certain compatibility conditions.
The primary goal of this article is it to generalize this description to the situation of $T$-equivariant vector bundles on a normal variety $X$ endowed with a faithful action by an algebraic torus $T$. In this situation, the objects corresponding to an equivariant vector bundle will be collections of filtrations  of vector bundles on a suitable quotient of $X$ by $T$.

We now establish the necessary notation to present our first main result. Unless otherwise specified, we will be working over an arbitrary algebraically closed field $K$.
\begin{definition}Let $Z$ be an Artin stack and $\xray$ a set.
	\begin{enumerate}
	  \item A $(Z,\xray)$-filtration consists of a vector bundle $\E$ on $Z$ together with, for every $\rho\in \xray$, a full decreasing filtration $$\ldots\E^\rho(i-1)\supseteq \E^\rho(i)\supseteq\E^\rho(i+1)\ldots$$ of $\E$ by sub-bundles, that is, for all $i\in\ZZ$, $\E^\rho(i)$ and $\E^\rho(i)/\E^\rho(i+1)$ are locally free. Being \emph{full} means that $\E=\E(i)$ for $i\ll 0$ and $\E(i)=0$ for $i\gg 0$.
		\item A \emph{morphism} between $(Z,\xray)$-filtrations $(\E,\{\E^\rho(i)\})$ and 
	$(\F,\{\F^\rho(i)\})$ 	is a map of vector bundles $\phi:\E\to\F$ such that for every $\rho\in \xray$ and every $i\in\ZZ$, $$\phi(\E^\rho(i))\subseteq\F^\rho(i).$$
\end{enumerate}
\end{definition}

Now, let $T$ be an algebraic torus. Consider a $T$-variety $X$, that is, a normal variety $X$ with a faithful $T$-action. Let $W\subset X$ be the $T$-invariant open subset consisting of those points with finite stabilizers.
The quotient stack $Z=[W/T]$ is a tame Artin stack with finite stabilizers \cite{tame}, and if $\Char K=0$, it is Deligne-Mumford \cite[7.17]{vistoli:89a}.  Let $\pi:W\to Z$ be the quotient map. 
There are finitely many $T$-invariant prime divisors $D_i$ of $X$ not meeting $W$; let them be indexed by a set $\xray$.

  In this situation, we may identify the class of so called \emph{$X$-compatible} $(Z,\xray)$-filtrations, see Definition \ref{def:compatible}. Essentially, on certain open subsets of $Z$ we impose the existence of decompositions of $\E$ into sums of line bundles, such that the elements of the filtrations are given by direct sums of some of these line bundles.

\begin{theorem}\label{thm:stacky}
There is a fully faithful functor $\bF$ which embeds the category of $T$-equivariant vector bundles on $X$ into the category of $X$-compatible $(Z,\xray)$-filtrations. If $X$ is factorial, this functor is an equivalence of categories.
\end{theorem}

We construct the functor $\bF$ in Section \ref{sec:gen}, where we also show that it is fully faithful. In Section \ref{sec:compatible}, we show that it is essentially surjective when $X$ is factorial, i.e. when every Weil divisor is  Cartier.
We actually prove a slightly generalized version of the above theorem to deal with equivariant bundles on the product of a $T$-variety with an arbitrary scheme, see Theorems \ref{thm:F} and \ref{thm:compatible}.

Many natural operations involving vector bundles correspond to easily describable operations on $(Z,\xray)$-filtrations: taking global sections (Section \ref{sec:gs}),  forming equivariant line bundles from invariant divisors (Section \ref{sec:lb}), and forming direct sums, tensor and symmetric and wedge products, and duals (Section \ref{sec:sc}).
In the case that $X$ is smooth, the tangent bundle $\T_X$ comes with a $T$-equivariant structure; the corresponding $(Z,\xray)$-filtration has a natural description involving extensions of the tangent bundle of $Z$, see Theorem \ref{thm:tangent}. This description can be made quite explicit for the case that $X$ is rational and $T$ acts with complexity one, that is, $\dim T=\dim X-1$. As a first application of our machinery, we compute global vector fields on smooth complexity-one Fano threefolds, see Example \ref{ex:fanotan}.

A second application of our machinery concerns the splitting behaviour of equivariant vector bundles on projective space. Kaneyama has shown that any \emph{toric} vector bundle of rank $r<n$ on $\PP^n$ splits as a sum of line bundles \cite{kaneyama:88a}. We generalize this result in Section \ref{sec:splitting} as follows:
let $m,n\in\NN$ with $1\leq m\leq  n$ and 
consider the $(K^*)^m$-action on $\PP^n$ such that the $i$th coordinate of $(K^*)^m$ acts on the $i$th homogeneous coordinate of $\PP^n$ by multiplication.

\begin{theorem}
	Any $(K^*)^m$-equivariant vector bundle on $\PP^n$ whose  rank is smaller than the mininum of $n$ and $m+3$  splits  as a direct sum  of line bundles.
\end{theorem}
\noindent Hartshorne has conjectured that for $n\geq 7$, any rank two vector bundle on $\PP^n$ splits \cite[Conjecture 6.3]{hartshorne:74a}. Our theorem shows that this conjecture is equivalent to showing that any rank two vector bundle on $\PP^n$, $n\geq 7$, admits a $K^*$-equivariant structure as above.

Many examples of $T$-equivariant bundles come from considering a toric vector bundle $\V$, but restricting the action of the big torus to that of a subtorus $T$. In this situation, the quotient $Z$ inherits an action by a quotient torus, and the corresponding $(Z,\xray)$-filtration is equivariant with respect to this action. It turns out that this data may be described combinatorially, and in Section \ref{sec:downgrades} we show how to translate the combinatorial data describing $\V$ as a toric vector bundle to the combinatorial data describing $Z$ and the corresponding $(Z,\xray)$-filtration.

As a final application of our machinery, we discuss families of $T$-equivariant vector bundles in Section \ref{sec:def}. 
Given a family of equivariant bundles corresponding to some family of filtrations on the quotient, the bundles appearing as fibers in the family correspond to fibers of the filtrations, see Proposition \ref{prop:family}. We use this to construct non-toric deformations of the tangent bundle on toric del Pezzo surfaces. Along the way, we give a combinatorial description of the obstruction space to deformations of the tangent bundle on smooth toric varieties, see Theorem \ref{thm:obstructions}.

It is worth noting that the category of equivariant vector bundles on a $T$-variety $X$ is tautologically equivalent to the category of vector bundles on the quotient stack $[X/T]$. However, our stack $[W/T]$ is in general much more nicely behaved then $[X/T]$; in particular, it has a coarse moduli space, and one can do  geometry on it. 
\subsection*{Acknowledgments} We would like to thank Piotr Achinger, Milena Hering, Andreas Hochenegger, Andrew Niles, Arthur Ogus, and Martin Olsson for helpful comments and conversations.
\section{Preliminaries}

\subsection{Stacks}
We briefly recall the necessary notions related to stacks, and motivate them. An excellent introduction is  \cite[Section 7]{vistoli:89a}, which we follow here. Readers familiar with the language of stacks may skip to Section \ref{sec:actions}. 

Let $\sch$ be the category of schemes over $K$.

\begin{definition}
  A \emph{category over $\sch$} is a category $\C$ together with a functor $p_\C:\C\to \sch$.
A morphism between two such categories $\C$, $\C'$ is a functor $\F:\C\to \C'$ such that $\F\circ p_{\C'}=p_\C$.
\end{definition}

\begin{ex}[{\cite[(7.2)]{vistoli:89a}}]
Any scheme $X$ determines a category over $\sch$, which we also denote by $X$. Indeed, consider the category $\schx$ of schemes over $X$, with the natural forgetful functor $\schx\to\sch$. We can recover $X$ as the image in $\sch$ of the terminal object of $\schx$. 
\end{ex}

We can thus embed the category of $K$-schemes in the two-category of categories over $\sch$.  The primary motivation for doing so comes from moduli problems: many important moduli  problems can be represented by a category over $\sch$, but not by an actual scheme. 
The typical category over $\sch$ is very far away from being a scheme, but there are special classes of categories over $\sch$ which exhibit increasingly scheme-like behaviour: stacks \cite[(7.3)]{vistoli:89a}, Artin stacks \cite[(7.14)]{vistoli:89a}, tame stacks \cite{tame}, and Deligne-Mumford stacks \cite[(7.14)]{vistoli:89a}.
For our purposes, the reader may assume that any stack is a quotient stack, which we define below.
Precise definitions of the other notions are found in the above references. 

\begin{definition}[{\cite[(7.17)]{vistoli:89a}}]
  Let $G$ be an algebraic group, and $X$ a scheme with regular $G$-action. We define the \emph{quotient stack} $[X/G]$ as follows: objects of $[X/G]$ are $G$-principal bundles $E\to Y$ with a $G$-equivariant morphism $E\to X$. Morphisms in $[X/G]$ are morphisms of $G$-bundles compatible with the maps to $X$. The functor $[X/G]\to \sch$ maps a $G$-principal bundle $E\to Y$ to $Y$. 
\end{definition}  

Such a quotient stack comes with a natural quotient morphism $\pi:X\to [X/G]$. Indeed, consider the trivial $G$-bundle over $X$, with equivariant structure map to $X$ given by the $G$-action. Then $\pi$ maps any $X$-scheme $Y\to X$ to the pullback of this bundle along $Y\to X$. 
This morphism models the quotient morphism $X\to X/G$ from the situation where $G$ acts freely on $X$. In fact, if $G$ acts freely, then $[X/G]$ is isomorphic to the scheme $X/G$. Furthermore, any time that $X$ admits a categorical quotient $X/G$ (as will be the case in the situation relevant to us), then there is a uniquely determined morphism $[X/G]\to X/G$ commuting with the quotient morphisms.

\begin{definition}
  Let $[X/G]$ be a quotient stack. A \emph{cover} for $[X/G]$ is a $G$-principal bundle $E\to Y$ together with a surjective $G$-equivariant map $E\to X$.
\end{definition}

Quasicoherent sheaves on a quotient stack $[X/G]$ are defined by giving, for each cover $E\to Y$, a quasicoherent sheaf on $Y$, along with isomorphisms between pullbacks of these sheaves satisfying a cocycle condition, see \cite[(7.18)]{vistoli:89a} for a precise definition.  

\subsection{Torus Actions}\label{sec:actions}
By $T$ we will always denote an algebraic torus with character lattice $M$. The dual lattice of one-parameter subgroups will be denoted by $N$.
By a $T$-variety, we will always mean a separated, integral, normal scheme of finite type over $K$ with a faithful $T$-action.
The \emph{complexity} of a $T$-variety $X$ is defined as $\dim X-\dim T$.
By a \emph{vector bundle} we simply mean a locally free sheaf of $\CO_X$-modules.

Let $G$ be an algebraic group which acts regularly on an Artin stack $X$. Recall that a $G$-equivariant sheaf on $X$ is a pair $(\F,\theta)$ consisting of a sheaf $\F$ on $X$ and an isomorphism $\theta: \pi^* \F \to \rho^* \F$ satisfying a natural cocycle condition, see \cite[Section 0.2]{bernstein:94a}. Here, $\pi$ is the projection $G\times X\to X$ and $\rho$ is the map $G\times X\to X$ sending $(g,y)$ to $g^{-1}y$. 
A morphism of equivariant sheaves is a morphism of the underlying sheaves which commutes with the respective isomorphisms.
The situation of interest to us, namely when $\F$ is a vector bundle, may be described geometrically as well. Indeed, let $\VV$ be the geometric vector bundle over a variety $X$ whose sheaf of sections is the locally free sheaf $\V$. Giving an equivariant structure on $\V$ is equivalent to giving a regular action of $G$ on $\VV$ which commutes with the projection to $X$ and the action on $X$, and which is linear on the fibers of $\VV$.

We now state some basic results we will need throughout the article.

\begin{prop}[cf. {\cite[Proposition 2.2]{payne:08a}}]\label{prop:cover}
Let $X$ be a $T$-variety, $S$ a  scheme, and $\V$ a $T$-equivariant vector bundle on $X\times S$. 
There exists an invariant affine open cover $\{U_\alpha\}$ of $X\times S$ such that the restriction of $\V$ to any $U_\alpha$ is  a direct sum of trivial equivariant line bundles.
\end{prop}
\begin{proof}
Let $x$ be the generic point of some $T$-orbit $\co$ in $X\times S$. Let $s_1,\ldots,s_k$ be semi-invariant sections of $\V_x$ such that $s_1(x),\ldots,s_k(x)$ form a basis of the vector space $\V_x\otimes K(x)$, where $K(x)$ is the residue field of the point $x$. The set of points $x'$ such that $s_1(x'),\ldots,s_k(x')$ is not a basis of $\V_{x'}\otimes K(x')$ is closed and $T$-invariant. Hence, $x$ is contained in an open $T$-invariant set $U$ on which $\V$ decomposes as the direct sum of the line bundles generated by the $s_i$. 

By Sumihiro's theorem \cite{sumihiro:74a}, $X$ and hence $X\times S$ have $T$-invariant affine covers.
Intersecting this cover with the above open invariant set $U$ gives us a quasi-affine invariant set $U'$ containing $x$ on which $\V$ decomposes. But any quasi-affine $T$-invariant set may be covered by invariant affines. To complete the proof, note that an arbitrary point $x'\in\co$ is contained in every invariant open subset which contains $x$.
\end{proof}

Consider a $T$-equivariant sheaf $\F$ on an Artin stack with trivial $T$-action. The subsheaf of semi-invariant sections of weight $u\in M$ will be denoted by $\F_u$. In particular, the subsheaf of $T$-invariant sections will be denoted by $\F_0$.
\begin{lemma}\label{lemma:stacky}
Let $T$ be a torus acting on a scheme $W$. Let $Z$ be the quotient stack $Z=[W/T]$ with quotient map $\pi:W\to Z$, and let $\mcL$ be a $T$-equivariant line bundle on $W$. 
\begin{enumerate}
	\item $\pi^*(\pi_*(\mcL)_0)$ is canonically isomorphic to $\mcL$.
	\item For every $u\in M$, $\pi_*(\mcL)_u$ is invertible.
	\item For every $u,u'\in M$, there is a canonical isomorphism $\pi_*(\mcL)_u\otimes\pi_*(\CO_W)_{u'}\cong \pi_*(\mcL)_{u+u'}$.
\end{enumerate}
\end{lemma}
\begin{proof}
	By the construction of $[W/T]$, see e.g. \cite[7.17]{vistoli:89a}, it suffices to show that each of the above claims holds when $W\to Z$ is a $T$-principal bundle. But in this case, all claims are well known. 
\end{proof}

Let $X$ be a $T$-variety, $W\subset X$ the open subset on which $T$ has finite stabilizers.

\begin{lemma}\label{lemma:tdiv}
Let $D$ be a prime $T$-invariant divisor in $X$ not meeting $W$. Then the stabilizer of $T$ at the generic point of $D$ is one-dimensional. If $\rho_D$ is the associated primitive co-character in $N$, then any rational function $f$ of weight $u$ vanishes to order $-u(\rho_D)$ on $D$.
\end{lemma}
\begin{proof}
Both claims follow from the proof of {\cite[Proposition 3.2]{hausen:10a}}.
\end{proof}
\begin{rem}\label{rem:xray}
Given any prime $T$-invariant divisor $D$ in $X$ not meeting $W$, we can thus associate the primitive lattice vector $\rho_D\in N$ corresponding to the isotropy group of a general point of $D$.
It follows from the second claim of the above lemma (and the fact that $X$ is separated) that if $D$ and $D'$ are two distinct divisors of this type, $\rho_D\neq \rho_{D'}$. 
Thus, we can index prime $T$-invariant divisors in $X$ not meeting $W$ by a subset $\xray$ of primitive elements of $N$. We will henceforth make this identification. Given $\rho\in\xray$, we will denote the corresponding divisor by $D_\rho$.
\end{rem}

\section{The functor $\bF$}
\label{sec:gen}
As in the introduction, let $X$ be a $T$-variety, and let  $W\subset X$ be the open subset consisting of those points with finite stabilizers. We call $W$ the \emph{DM-locus} of $X$.
Consider the quotient stack $Z=[W/T]$. 
$T$-invariant prime divisors of $X$ not meeting $W$ may be indexed by a collection $\xray$ of primitive vectors in $N$, see Remark \ref{rem:xray}. The aim of the section is to construct the functor $\bF$ of Theorem~\ref{thm:stacky}. We will actually do this in a generalized setting to obtain a relative version for equivariant vector bundles on products $X \times S$, where $S$ is an arbitrary  scheme. In this situation we denote the quotient map of the product $W \times S \to Z \times S$ by $\pi$.
The case of primary interest for us will remain that where $S=\Spec K$.
Note that $Z\times S$ is still an Artin stack, so we may talk about $(Z \times S,\xray)$-filtrations.

\begin{theorem}\label{thm:F}
There is a fully faithful functor $\bF$ which embeds the category of $T$-equivariant vector bundles on $X \times S$ into the category of $(Z \times S,\xray)$-filtrations. 
\end{theorem}

Let $\V$ be an equivariant vector bundle on $X\times S$. 
Set $\E=\pi_*(\V|_{{W\times S}})_0$.
Now, for any $\rho\in \xray$, let 
$\V^\rho$ denote the subsheaf of $\V|_{W\times S}$ consisting of those sections $s$ which extend to some open subset $U\subset X\times S$ intersecting $D_\rho\times S$.
\begin{lemma}\label{lemma:stackfil1}
For any $u\in M$, the sheaf $\pi_*(\V^\rho)_u$ is locally free.
\end{lemma}
\begin{proof}
The claims are local on $Z\times S$, so by Proposition \ref{prop:cover} we may assume that we are working on an affine subscheme $Y$ of $X\times S$  containing  a general point of $D_\rho$ on which $\V$ decomposes as a direct sum of equivariant rank one free sheaves. We may thus reduce to the case that $\V$ is isomorphic to $\CO_X$, generated by a semi-invariant section $s$ of weight $v\in M$. 
All sheaves we henceforth consider will be sheaves on $Y\cap(W\times S)$ or its image under $\pi$.

We now show that in this situation, $\pi_*(\V^\rho)_u$ is either $0$ or $\pi_*(\V)_u$.  Let $s'\in\V$ be a semi-invariant section of weight $u$. It may be written as $s'=s\cdot f$ for some semi-invariant function $f\in\CO_W$ of weight $u-v$. Then the order of vanishing of $s'$ at the generic point of $D_\rho\times S$ is given by the integer $(v-u)(\rho)$. Indeed, the valuation defined by the divisor $D_\rho\times S$ is just as in  Lemma \ref{lemma:tdiv}.
Hence, the degree $u$ part of $\V^\rho$ either vanishes, or equals that of $\V$.
In the first case, $\pi_*(\V^\rho)_u$ is $0$; in the second,  $\pi_*(\V^\rho)_u=\pi_*(\V)_u$.
The lemma now follows directly from Lemma \ref{lemma:stacky}.
\end{proof} 
Now for any $u\in M$, there is a natural inclusion 
$$\E^\rho(u):=\pi_*\big(\V^\rho\big)_u\otimes \pi_*\big(\CO_{W\times S}\big)_{-u}\hookrightarrow\E$$
obtained by multiplication of sections. 
For any $i\in\ZZ$, consider some $u\in M$ such that $u(\rho)=i$. Let $\E^\rho(i)$ be the image of 
$\E^\rho(u)$
in $\E$.

\begin{lemma}\label{lemma:stackfil2}
The vector bundle $\E^\rho(i)$ doesn't depend on the choice of $u$. For any $i$, $\E^\rho(i+1)$ is a sub-bundle of $\E^\rho(i)$. 
\end{lemma}
\begin{proof}
From the proof of \ref{lemma:stackfil1}, we have seen that the vanishing of $\E^\rho(u)$ only depends on $u(\rho)$, not on $u$. The first claim follows from the canonical isomorphism
$$\pi_*(\V|_{W\times S})_u\otimes\pi_*(\CO_{W\times S})_{-u}\cong \pi_*(\V|_{W\times S})_{u'}\otimes\pi_*(\CO_{W\times S})_{-u'}$$
for any $u,u'\in M$, see Lemma \ref{lemma:stacky}. The second claim follows from this isomorphism, the reduction used in the proof of Lemma \ref{lemma:stackfil1}, and the fact that if $\E^\rho(u)$ does not vanish, neither will $\E^\rho(u')$ for any $u'$ with $(u-u')(\rho)\geq 0$. 
\end{proof}

The functor $\bF$ associates to $\V$ the $(Z \times S,\xray)$-filtration $(\E,\{\E^\rho(i)\})$. Morphisms of equivariant vector bundles map to morphisms of $(Z \times S,\xray)$-filtrations in an obvious manner.
\begin{rem}
\label{sec:rem-alternativ-V-rho}
One may equivalently define $\E^\rho(i)$ as $\pi_*(\V^\rho(i))_0$, where $\V^\rho(i)$ is the subsheaf of $\V|_{W\times S}$ of sections which vanish to order at least $i$ along $D_\rho$. 
\end{rem}

\begin{rem}
\label{sec:rem-rank-one-filt}
  If $\V$ has rank one, then the filtration $\E^\rho(i)$ is determined by the unique integer $j$  such that $\E^\rho(j)=\E$ and $\E^\rho(j+1)=0$. If $s$ is a local generator of $\V$ in a neighbourhood of $D_\rho\times S$ of weight $v$, then this jumping position is given by $j=-v(\rho)$. Indeed, we are looking for the products $f \cdot s$ of weight $0$, i.e. $\deg(f) = -v$. However,  the vanishing order of $f \cdot s$ is given by $-v(f)$  due to Lemma \ref{lemma:tdiv}.
\end{rem}

\begin{proof}[Proof of Theorem~\ref{thm:F}]
	By Lemmas \ref{lemma:stackfil1} and \ref{lemma:stackfil2}, the functor $\bF$ is well-defined. We now show that it is fully faithful.
The fundamental observation we will use is that if $\V$ is an equivariant vector bundle and  $(\E,\{\E^\rho(i)\})$ the corresponding $(Z\times S,\xray)$-filtration, then $\V|_{W\times S}=\pi^*(\E)$ by Lemma \ref{lemma:stacky}.

Now, let $\V$, $\V'$ be equivariant vector bundles and consider equivariant maps $\phi_1,\phi_2:\V\to\V'$. If the corresponding maps $\E\to\E'$ are equal, then so are the induced maps
from $\V|_{W\times S}=\pi^*(\E)$ to $\V'|_{W\times S}=\pi^*(\E')$ obtained by pullback. But these are the same maps as those induced by $\phi_1$ and $\phi_2$. Hence, $\phi_1$ and $\phi_2$ are equal on $\V|_{W\times S}$. But then they must agree on all of $X\times S$, since $\V$ is a vector bundle and $W$ is dense in $X$. This shows that $\bF$ is faithful.

We now show that $\bF$ is full. Indeed, let $\V$, $\V'$ be equivariant vector bundles with corresponding $(Z\times S,\xray)$-filtrations $(\E,\{\E^\rho(i)\})$ and $(\E',\{{\E}'^\rho(i)\})$. Let $\psi:\E\to\E'$ be a map of $\CO_{Z\times S}$-modules inducing a map of filtrations. We claim that $\psi$ corresponds to a map $\V\to\V'$.
Now, the pullback of $\psi$ along $\pi$ gives an equivariant map 
$\pi^*(\psi):\V|_{W\times S}\to\V'|_{W\times S}$. We will show that this extends to all of $X\times S$. 
Without loss of generality, we may assume that $S$ is affine, say $S=\Spec A$.

Let $U_\rho=\Spec B$ be an open affine subset of $X$ containing only points of $W$ and $D_\rho$, and (after possibly shrinking $U_\rho)$ assume that $D_\rho$ is cut out by $f\in B$.
The sheaf $\V'$ restricted to $U_\rho\times S$ corresponds to an $B\otimes A$-module $M$.
Hence, the restriction to $(U_\rho\cap W)\times S$ corresponds to the $B_f\otimes A$-module $M_f$. Consider a section $s$ of $\V$ on $U_\rho\times S$; then $\pi^*(\psi)(s)\in M_f$. But since the filtrations $\E^\rho(i)$ and ${\E'}^\rho(i)$ control the order of vanishing along $D_\rho$ and $\psi$ respects them, $s$ must actually be in $M$. Hence, $\pi^*(\psi)$ extends to any such set $U_\rho\times S$.

Now, we can conclude that $\pi^*(\psi)$ extends to $X'\times S$, where $X\setminus X'$ has codimension greater than one in $X$. However, since $X$ is normal and $\V$, $\V'$ are locally free, it follows that $\pi^*(\psi)$ extends to a map defined on all of $X\times S$.
\end{proof}
 
\begin{rem}
 Suppose that $X$ is an affine $T$-variety of complexity $k$ and assume that $\Char K=0$. In \cite{altmann:06a}, Altmann and Hausen show that $X$ can be described in terms of a \emph{polyhedral divisor} $\D$ on a $k$-dimensional normal variety $Y$. Instead of allowing just integral or rational coefficients, the coefficients of $\D$ are polyhedra all having some common recession cone. Many aspects of the geometry of $X$, including its $T$-orbit structure, can be read off the combinatorics of the coefficients of $\D$.

 This construction is related to our functor $\bF$. The variety $X$ is affine, so we can recover it from the global sections of the $T$-equivariant bundle $\CO_X$. 
 Since $W\to Z=[W/T]$ is a $T$-principal bundle, we have 
 \begin{equation*}
 	\pi_*(\CO_W)=\bigoplus_{u\in M} \mcL^u\chi^u \qquad\qquad \mcL^u=\CO\left(\sum_v u(v)D_v\right)
 \end{equation*}
 where the sum on the right is over some finite set of lattice points $v\in N$, and each $D_v$ is a Cartier divisor on $Z$. Hence, $W$ corresponds to a divisor on $Z$ whose coefficients are lattice points of $N$. To recover $\CO_X$ (more precisely its global sections, see  Proposition \ref{prop:gs}), we only need the additional information of the set $\xray$. In this affine situation, the elements of $\xray$ will all be rays of the positive hull of $\xray$.
 
 We thus see that the complete information of $X$ is encoded by a divisor $\D'$ on the (possibly non-separated) Deligne-Mumford stack $Z$ whose coefficients are lattice translates of a rational polyhedral cone $\sigma$.
On the other hand, the polyhedral divisor $\D$ is a divisor on the variety $Y$, but its coefficients $\D_P$ may have numerous vertices, and need not be lattice polyhedra. However, the recession cone of each coefficient $\D_P$ is $\sigma$. 
Furthermore, suppose the map $\D:M\to \WDiv_\QQ Y$ 
defined by 
$$
\D(u)=\sum \min_{v\in\D_P} u(v) P
$$
is linear and $\D(u)$ is Cartier for all $u$.
Then the coefficients of $\D$ are also just lattice translates of $\sigma$,  
$Z$ is a variety, and $\D'$ is a polyhedral divisor describing $X$.
\end{rem}

\section{Compatibility Conditions}\label{sec:compatible}
In this section, we describe the image of the functor $\bF$ from Theorem \ref{thm:F} under certain assumptions.
As before, let $X$ be a $T$-variety with DM-locus $W$, and $Z=[W/T]$ the quotient stack with quotient map $W\to Z$. Let $\xray$ be the subset of $N$ corresponding to the invariant prime divisors not meeting $W$.

For any $T$-orbit $\co\subset X$, let $Z_\co$ be the image in $Z$ of those orbits $\co'$ in $W$ whose closures in $X$ contain $\co$.
Let $\xray_\co$ consist of those $\rho\in\xray$ such that $D_\rho$ contains $\co$, and let $M_\co$ denote the lattice of integral linear forms on the span of $\xray_\co\subset N$.

Let $S$ be any scheme. As before, we will also denote the quotient map $W\times S\to Z\times S$ by $\pi$.
\begin{definition}\label{def:compatible}
	A $(Z\times S,\xray)$-filtration $(\E,\{\E^\rho(i)\})$ is \emph{$X$-compatible} if for every 
	$T$-orbit $\co\subset X$ and every closed point $s\in S$
there exists an open set $U\subset Z\times S$ containing $Z_{\co}\times s$ and a decomposition
\begin{equation*}\E|_{U}=\bigoplus_{[u]\in M_\co} \E_{[u]}
\end{equation*}
such  that each $\E_{[u]}$ is a sum of invertible $\CO_U$-modules, and
for every $\rho\in\xray_\co$,
\begin{equation*}\E|_{U}^\rho(i)=\bigoplus_{\substack{
u(\rho)\geq i}} \E_{[u]}.
\end{equation*}
\end{definition}
In order to completely describe the image of $\bF$, we will need to impose a slightly technical condition on $X\times S$:
\begin{equation*}\label{dagger}
	\begin{array}{ p{13cm}}
					\textrm{For any open, affine, invariant subscheme $Y\subset X\times S$ and any equivariant line bundle $\mcL$ on $Y\cap W$, there exists an equivariant line bundle $\overline \mcL$ on $Y$ restricting to $\mcL$ which is locally generated by semi-invariant sections whose weights are orthogonal to all $\rho\in \xray$ with $D_\rho\cap Y\neq \emptyset$.
					}\\
\end{array}
\tag{$\dagger$}\end{equation*}
This condition is satisfied if $X$ is factorial and $S$ is a point. Indeed, we may identify line bundles on $X$ with Cartier divisors, and to extend a divisor to $Y$ we just take its closure. This condition is also satisfied whenever $X$ is a toric variety and $T$ is a maximal torus of $\aut(X)$. In this case, the quotient map $Y\cap W\to[Y\cap W/T]$ extends to a map $Y\to[Y\cap W/T]$, and $\overline \mcL$ may be constructed as the pullback of the invariant pushforward of $\mcL$.
\begin{theorem}\label{thm:compatible}
Let $X$ be a $T$-variety and $S$ a scheme.
\begin{enumerate}
  \item  For any equivariant bundle $\V$ on $X\times S$, $\bF(\V)$ is $X$-compatible. \label{thm:1}
  \item If $X$ satisfies \eqref{dagger}, then $\bF$ induces an equivalence of categories between
    $T$-equivariant vector bundles on $X\times S$ and $X$-compatible $(Z\times S,\xray)$-filtrations.\label{thm:2}
\end{enumerate}
\end{theorem}
\begin{proof}
Consider any $T$-equivariant vector bundle $\V$ on $X\times S$.
From  Proposition \ref{prop:cover}, $\V$ locally splits as 
$$
\V=\bigoplus_{[u]} \V_{[u]}
$$
with each $\V_{[u]}$ further splitting in a sum of equivariant line bundles $\bigoplus_i \mcL_{[u]}^i$ generated by invariant sections in weights mapping to $[u]$. Then $\pi_*({\V_{[u]}}|_{W\times S})_0$ gives the desired local splittings of $\pi_*(\V|_{W\times S})$ satisfying the compatibility condition. This shows part \ref{thm:1} of the theorem.
  
For part \ref{thm:2}, we suppose that $X$ satisfies $\eqref{dagger}$.  Let $(\E,\{\E^\rho(i)\})$ be an $X$-compatible $(Z\times S,\xray)$-filtration. We will construct a corresponding vector bundle on $X$ via a local construction which  glues. Without loss of generality, we will assume that $S$ is affine.

Fix a $T$-orbit $\co$ of $X$, a closed point $s$ of $S$,  and let $U$ and $\E_{[u]}$  be as in Definition \ref{def:compatible} above. 
	Let $V=\pi^{-1}(U)$ and let $Y$ consist of the closures of the orbits in $V$; note that $Y$ is an open subscheme of $X\times S$, and contains $\co\times s$.
After possibly shrinking $V$, we may assume that $Y$ is $T$-invariant, open, affine, and contains $\co\times s$, with $V=Y\cap W$.

	Define the sheaf
	$\F_{[u]}$ on $V$ as the pullback of the sheaf $\E_{[u]}$ on $U$.
	Due to Definition \ref{def:compatible}, this decomposes as a direct sum of equivariant line bundles
$$
\F_{[u]}=\bigoplus \mcL_i.
$$
We extend this to an equivariant vector bundle on $Y$ via
$$
\overline \F_{[u]}=\chi^{-u}\otimes \bigoplus \overline \mcL_i
$$
where the $\overline \mcL_i$ are as in \eqref{dagger}.
Now let $\F=\bigoplus \overline \F_{[u]}$. It follows from our construction that the image of the vector bundle $\F$ under $\bF$ is the restriction of $(\E,\{\E^\rho(i)\})$ to $U$, where we consider only those filtrations indexed by $\rho\in\xray_\co$.

Varying $\co$ and $s$, we now define a vector bundle  on $X\times S$ by gluing the sheaves $\F$ constructed above. Indeed, since locally the images of the corresponding vector bundles under the functor $\bF$ agree, these sheaves satisfy the necessary conditions for gluing by Theorem \ref{thm:stacky}.
Hence, we may conclude that the set of $X$-compatible $(Z\times S,\xray)$-filtrations lies in the image of $\bF$.

\end{proof}

If we do not impose condition \eqref{dagger}, not every $X$-compatible $(Z\times S,\xray)$-filtration needs to lie in the image of $\bF$. Indeed, consider the following example:

\begin{ex}
 \label{ex:not-surjective} 
 Consider the affine variety $X\subset \A^3$ given by the equation $f=x_1x_2-x_3^2$, with the $K^*$-action given by the grading $\deg x_1 = - \deg x_2 = 1$ and $\deg x_3 =0$. The variety $X$ has an $A_1$-singularity at the origin, which is non-factorial. In fact, we will see that \eqref{dagger} is not fulfilled. We have $W = X \setminus \{0\}$. The quotient $[W/K^*]$ is $\AA^1$ with a doubled point at the origin. The quotient map is given by the inclusion $K[x_3] \hookrightarrow K[x_1,x_2,x_3]/(f)$, but the orbits $\co_1 = \{(x,0,0 \mid x \in K^*)\}$ and $\co_2 = \{(0,y,0) \mid y \in K^*\}$ are mapped to the two different instances of the origin $0_1$ and $0_2$, respectively.

Now, we consider the line bundle $\E := \CO([0_1])$ on $Z$. Since $\xray = \emptyset$ this defines a $(Z,\xray)$-filtration and   the $X$-compatibility condition is automatically fulfilled. Indeed,  $\E$ is a line bundle. But $\pi^*\E = \CO(\co_1)$ does not extend to a line bundle on $X$, since $\overline \co_1$ is not Cartier at the origin. This shows that the  $(Z,\xray)$-filtration given by $\E$ does not lie in the image of $\bF$.
\end{ex}



\section{The toric situation}
\label{sec:toric-situation}
In this section we consider the situation of a toric variety $X$ defined by a fan $\Sigma$ living in $N \otimes \QQ$. As mentioned above, \eqref{dagger} is satisfied in this case. The only orbit whose points have finite stabilizers is the embedded torus $T$ itself, and the prime divisors outside of $W$ are indexed by the rays in $\Sigma$, which we denote by $\Sigma^{(1)}$.
Now, $Z=[W/T]$ is just a point. Hence, we consider filtrations $E^\rho(i)$ of some vector space $E \cong K^r$ indexed by $\xray = \Sigma^{(1)}$.

 The $T$-orbits correspond to cones in the fan $\Sigma$ and an orbit $\co(\sigma)$ is contained in $D_\rho$ if and only if the ray of $\rho$ is contained in the cone $\sigma$. Now, $M_{\co(\sigma)} = M/\sigma^\perp$ holds and our compatibility condition translates as follows. For every maximal cone $\sigma$ we have a decomposition
\begin{equation*}E=\bigoplus_{[u]\in M/\sigma^\perp} E_{[u]}
\end{equation*}
such  that for every ray $\rho \in \sigma(1)$ we have
\begin{equation*}E^\rho(i)=\bigoplus_{\substack{
u(\rho)\geq i}} E_{[u]}.
\end{equation*}
Here, we identify the ray $\rho$ with its primitive lattice generator. In the non-singular case the primitive generators $\rho$ of rays in $\sigma$ can be completed to a lattice basis of $N$  and the above conditions then just mean that for every cone $\sigma \in \Sigma$ there is a basis for $E$ such that all elements of the corresponding filtrations are spanned by some subset of this basis.
Hence, we rediscover Klyachko's classification result. 
\begin{theorem}[{\cite[Theorem~0.1.1]{klyachko:89a}}]
The category of toric vector bundles on a smooth toric variety $X_\Sigma$ is equivalent to the category of vector spaces $E$ with a collection of decreasing filtrations $E^\rho(i)$, $\rho \in \Sigma^{(1)}$, such that for any cone $\sigma \in \Sigma$ the filtrations  $E^\rho(i)$, $\rho \in \sigma(1)$ consist of coordinate subspaces for some basis of $E$.
\end{theorem}

\begin{rem}
   More generally, for $S$ an arbitrary scheme, Payne describes equivariant vector bundles on $X\times S$ in terms of filtrations of vector bundles on $S$ satisfying a rank condition
 \cite[Proposition 3.13]{payne:08a}. We leave it to the reader to check that Payne's rank condition is equivalent to our compatibility condition.
\end{rem}

\section{Global Sections and Standard Constructions}
As before, let $X$ be a $T$-variety with DM-locus $W$,  $Z=[W/T]$ the quotient stack, and let $\xray$ parametrize invariant prime divisors of $X$ not meeting $W$.
In this section, we describe how to calculate the global sections of an equivariant vector bundle on $X$ in terms of its $(Z,\xray)$-filtration, and how certain standard constructions of equivariant vector bundles translate to filtrations. This may all be generalized in a straightforward fashion to bundles on $X\times S$, for $S$ an arbitrary scheme; we leave the details to the reader.
\subsection{Global Sections}\label{sec:gs}
Let $\V$ be a $T$-equivariant vector bundle on  $X$ with corresponding $(Z,\xray)$-filtration $(\E,\{\E^\rho(i)\})$. It is straightforward to recover the graded components of $H^0(X,\V)$ from the filtration $(\E,\{\E^\rho(i)\})$.

For each $\rho\in\xray$, let $i_\rho$ be the largest $i\in \NN$ such that $\E^\rho(i)\neq 0$. Let $\Box_\V\subset M_\QQ$ be the polyhedron defined by 
$$
\Box_\V=\left\{u\in M_\QQ\ |\ u(\rho)\leq i_\rho\ \forall \rho\in\xray\right\}.
$$
Using notation as above, for any $u\in M$ let $\CO_Z(u)=\pi_*(\CO_W)_u$.
\begin{prop}\label{prop:gs}
Let $u\in M$. Then 
$$
H^0(X,\V)_u=\begin{cases}
H^0\left(Z,\E\otimes \CO_Z(u)\right) & \xray=\emptyset\\
\bigcap_{\rho\in\xray} H^0\left(Z,\E^\rho(u(\rho))\otimes \CO_Z(u)\right) & \xray\neq \emptyset
\end{cases}.
$$
If $u\notin\Box_\V$, then $H^0(X,\V)_u=0$.
\end{prop}
\begin{proof}
	Follows immediately from construction of $(\E,\{\E^\rho(i)\})$ in Section \ref{sec:gen}.
\end{proof}

\subsection{Line Bundles}\label{sec:lb}
Let $D$ be a $T$-invariant Cartier divisor on $X$. Then $D=D'+\sum_{\rho\in\xray} a_\rho D_\rho$ where the components of $D'$ all meet $W$. The line bundle $\CO_X(D)$ corresponds to the $(Z,\xray)$-filtration
$(\E,\{\E^\rho(i)\})$, where $\E=\CO(\pi_*(D'|_W))$ and 
\begin{align*}
	\E^\rho(i)=\begin{cases}
\E & i \leq -a_\rho \\
0 & i > -a_\rho
	\end{cases}.
	\end{align*}

\subsection{Sums, Products, and Duals}\label{sec:sc}
Let $\V$ and $\W$ be two equivariant vector bundles on $X$ with corresponding filtrations
$(\E,\{\E^\rho(i)\})$ and $(\F,\{\F^\rho(i)\})$. It is straightforward to check that
\begin{align*}
&\bF(\V\oplus\W)&&=(\E\oplus \F,\{\E^\rho(i)\oplus\F^\rho(i)\}); \\
&\bF(\V\otimes \W)&&=(\E\otimes \F,\{\G^\rho(i)\}),\qquad&\textrm{where}\qquad \G^\rho(i)&=\sum_{s+t=i}\E^\rho(s)\otimes\F^\rho(t);\\
&\bF(\wedge^k \V)&&=(\wedge^k \E,\{\G^\rho(i)\}),\qquad&\textrm{where}\qquad \G^\rho(i)&=\sum_{\substack{s_1,\ldots,s_k\\\sum s_j=i}}\E^\rho(s_1)\wedge\ldots\wedge\E^\rho(s_k);\\
&\bF(\Sym^k \V)&&=(\Sym^k \E,\{\G^\rho(i)\}),\qquad&\textrm{where}\qquad \G^\rho(i)&=\sum_{\substack{s_1,\ldots,s_k\\\sum s_j=i}}\E^\rho(s_1)\cdot\ldots\cdot\E^\rho(s_k);\\
&\bF(\V^*)&&=(\E^*,\{\G^\rho(i)\}),\qquad&\textrm{where}\qquad \G^\rho(i)&= (\E/\E^\rho(1-i))^* .\\
\end{align*}

\section{Tangent Bundles}
Let $X$ be a smooth $T$-variety. Then the tangent sheaf $\T_X$ is locally free, and since it is defined functorially, it comes with a natural equivariant structure. In the following, we will  describe the corresponding $(Z,\xray)$-filtration, and then use this to calculate global vector fields.

\subsection{General Description}
We first recall some details about Atiyah extensions:
Let $\mcL$ be a line bundle on a smooth tame Artin stack $Z$ and let $E=\Spec_Z \bigoplus_{u\in\ZZ} \mcL^u\chi^u$ be the corresponding $\GG_m$-torsor with projection $\pi:E\to Z$. Then there is an exact sequence
$$
\begin{CD}
0@>>> \CO_Z @>>> \pi_*(\T_E)_0 @>>> \T_Z @>>>0.
\end{CD}
$$
The corresponding class $\eta(\mcL)\in\Ext^1(\T_Z,\CO_Z)\cong H^1(Z,\Omega_Z)$ is called the Atiyah class of $\mcL$. 
Note that $\eta$ is simply the image of $\mcL$ under the group homomorphism $H^1(Z,\CO_Z^*)\to H^1(Z,\Omega_Z)$ induced by the map $\CO_Z^*\to \Omega_Z$ sending $f$ to $df/f$ \cite[cf. 3.3.3]{sernesi:06a}. Here, $d:\CO_Z\to\Omega_Z$ denotes the canonical derivation.

Now let $X$ be a smooth $T$-variety, and $W\subset X$ its DM-locus.
For $u\in M$, let $\mcL^u=\pi_*(\CO_W)_{-u}$; then 
$W=\Spec_Z \bigoplus_{u\in M} \mcL^{u}\chi^u$.

\begin{prop}\label{prop:atiyah}
If $X$ is smooth, there is an exact sequence
\begin{align}\label{eqn:atiyah}
\begin{CD}
0@>>> N\otimes \CO_Z @>\iota >> \pi_*(\T_W)_0 @>\phi>> \T_Z @>>>0.
\end{CD}
\end{align}
and the corresponding class $$\eta\in\Ext^1(\T_Z,N\otimes\CO_Z)\cong \Hom(M,H^1(Z,\Omega_Z))$$ is given by the map $u\mapsto \eta(\mcL^u)$.  
\end{prop}
\begin{proof}
It suffices to prove the claims in the case that $W$ is a principal $T$-bundle over a smooth variety $Z$. Without loss of generality, we may assume that the $\mcL^u$ are subsheaves of $K(Z)$.
Let $U\subset Z$ be any open affine set on which $W$ and $\T_Z$ trivialize. Then there exists $f\in K(Z)$ and $w\in N$ such that $\mcL^u$ is generated by $f^{u(w)}$ on $U$. 
$\pi_*(\T_W)_0$ is then  generated by the invariant derivations
\begin{align*}
	\partial_v:&gf^{u(v)}\chi^u\mapsto u(v)gf^{u(w)}\chi^u\qquad &v\in N\\
	\widetilde{\alpha}:&gf^{u(w)}\chi^u\mapsto \alpha(g)f^{u(w)}\chi^u\qquad &\alpha\in\T_Z(U). 
\end{align*}
The map $\iota$ sends $v\otimes g$ to $g\partial_v$; $\phi$ sends $\delta_v$ to 0 and $\widetilde{\alpha}$ to $\alpha$. This defines the desired exact sequence.

The above exact sequence determines, for any $u\in M$, an  injective map
$\CO_Z\to \pi_*(\T_W)_0$ by sending $g$ to $g\sum u(v)\partial_v$. The image of this map is naturally contained in $\pi'_*(\T_{W'})_0$, where $W'=\Spec_Z \bigoplus_{k\in\ZZ} \mcL^{ku}\chi^{ku}$ with projection $\pi':W'\to Z$; this sequence can be completed to the extension
$$
\begin{CD}
	0@>>> \CO_Z @>\iota >> \pi_*(\T_{W'})_0 @>\phi>> \T_Z @>>>0.
\end{CD}
$$
corresponding to $\eta(\mcL^u)$.
Thus, the extension in \eqref{eqn:atiyah} corresponds to the class $u\mapsto \eta(\mcL^u)$.
\end{proof}

\begin{theorem}\label{thm:tangent}
	Let $X$ be a smooth $T$-variety. Then the image of $\T_X$ under $\bF$ is given by the $(Z,\xray)$ filtration $(\E,\{\E^\rho(i)\})$, where 
$\E$ corresponds to the extension class $u\mapsto \eta(\mcL^u)\in \Ext^1(\T_Z,N\otimes \CO_Z)$ and 
\begin{align*}
	\E^\rho(i)=\begin{cases}
\E & i\leq 0\\
\iota(\langle \rho \rangle \otimes \CO_Z) & i=1\\
0 & i>1
	\end{cases}
\end{align*}
with $\iota(\langle \rho \rangle \otimes\CO_Z)$ the image of $\langle \rho \rangle \otimes \CO_Z$ in $\E$ under $\iota$. 
\end{theorem}
\begin{proof}
	Once again, it suffices to show the claim when $W$ is a principal $T$-bundle over a smooth variety $Z$.
	Let $U$ be an open affine set of $Z$ on which $W$ and $\T_X$ trivialize as in the proof of Proposition \ref{prop:atiyah}; we restrict our attention to $W|_{U}$ and use notation as above. 
	
	Let $u\in M$ and $\rho\in \xray$. Consider any semi-invariant function $s$ of weight $u$. Then $s\partial_v$ extends to $D_\rho$ if and only if
for any semi-invariant regular function $t$ of weight $u'$ on $W|_{U}$ extending to $D_\rho$, either $\partial_v t=0$ or $\rho(u)-\rho(u')\leq 0$.
Hence, if $\rho(u)\leq 0$, $s\partial_v$ extends to $D_\rho$. If $\rho(u)=1$,
$s\partial_v$ extends to $D_\rho$ if and only if $u'(v)=0$ for any $u'\in \rho^\perp$, that is, if $v$ is in the span of $\rho$.
Finally, if $\rho(u)\geq 2$, $s\partial_v$ extends to $D_\rho$ if and only if $u'(v)=0$ for any $u'$ satisfying $0\leq u'(\rho) \leq u(\rho)$, but this is impossible.
The claim now follows from Proposition \ref{prop:atiyah}.
\end{proof}
\begin{rem}
	The cotangent bundle of $X$ may be described by combining the above theorem with the discussion of Section \ref{sec:sc}.
	This leads to an even simpler description of the canonical bundle on a smooth  $T$-variety $X$. Indeed, if $n=\dim X$ and $m=\dim T$, $\wedge^n\pi_*(\Omega_W)_0\cong \wedge^m (M\otimes \CO_Z)\otimes \wedge^{n-m}\Omega_Z\cong \omega_Z$. Hence, $\omega_X$
	corresponds to the $(Z,\xray)$-filtration $(\E,\{\E^\rho(i)\})$	with $\E=\omega_Z$ and 
	$$
	\E^\rho(i)=\begin{cases}
\omega_Z& i\leq -1\\
0 & i>-1
	\end{cases}.
	$$
	Combined with \ref{sec:lb}, this gives a result similar to that of \cite[Theorem 3.21]{petersen:11a}.
\end{rem}

\subsection{Vector Fields}
Using the above description of the tangent bundle, Proposition \ref{prop:gs} allows one to calculate its global sections.
We now make the formula more explicit in a specific case, namely for complete rational complexity-one $T$-varieties.
We shall first assume that $Z$ is separated. In this case, $Z$ is a weighted projective line. That is, we may view $Z$ simply as a projective line on which divisors with some fixed rational denominators are allowed at a finite number of points.

Under these assumptions, $$W=\Spec_{\PP^1} \bigoplus_{u\in M} \mcL^u\chi^u $$ where 
$$\mcL^u=\CO_{\PP^1}\left(\sum_{P\in\mcP} \floor{u(v_P)}P\right)$$
such that
$\mcP$ is some finite subset of points in $\PP^1$ and for each $P\in\mcP$, $v_P\in N_\QQ$.
If the number of points $P$ with $v_P\notin N$ exceeds 2, then $\E=\pi_*(\T_W)_0$ will not split, making calculations somewhat tedious. 
However, 
since the coarse moduli map $\psi:Z\to \PP^1$ is exact and preserves cohomology, we can push forward twists of $\pi_*(\T_W)_0$ by $\CO_Z(u)$ to $\PP^1$ and easily calculate global sections there, since all bundles split.
Note that the pushforward of $\CO_Z(u)$ is just $\mcL^{-u}$.

For each $P\in\mcP$, let $n_P$ be the smallest natural number such that $n_Pv_P\in N$. Define for $u\in M$, $P\in\mcP$
\begin{align*}
\alpha_P(u)&=-u(v_P)+\frac{1-n_P}{n_P},&\qquad
\gamma(u)&=\sum_{P\in\mcP}\floor{\alpha_P(u)}P.
\end{align*}
Let $\mcP(u)$ consist of those $P\in\mcP$ such that $\alpha_P(u)\in\ZZ$, and let $\overline\mcP(u)=\mcP\setminus\mcP(u)$.
Set $\bv(u):=\sum_{P\in\mcP(u)} v_P$ and, if $\bv(u)\neq0$, let $V(u)$ be any codimension-one subspace of $N_\QQ$ not containing $\bv(u)$.

For any $v\in N_\QQ$, let $\partial_v:K(X)\to K(X)$  be the invariant rational derivation sending $f\cdot \chi^u\to u(v) f \cdot \chi^u$ for any $u\in M$, $f\in K(\PP^1)$. Likewise, let $\partial_y$ be the invariant rational  derivation sending $f\cdot \chi^u\to  \frac{\partial f}{\partial y} \cdot \chi^u$ for any $f\in K(\PP^1)$, where $y$ is a local parameter of $0\in\PP^1$ with a single pole at $\infty$.

Now, set
\begin{align*}
	e(u):&=\partial_y+\sum_{P\in\mcP(u)\setminus\{\infty\}} \frac{1}{y-P} \partial_{v_P}-
	\varepsilon\partial_{\bv(u)}
\end{align*}
where if $\overline\mcP(u)=\emptyset$, $\varepsilon=0$, and otherwise $\varepsilon=\frac{1}{y-P}$ for any $P\in \overline\mcP(u)\setminus\infty$.

\begin{prop}\label{prop:gstan}
	Let $u\in M$. Suppose $u(\rho)<1$ for all $\rho\in\xray$, and $\mcP(u)=\mcP$, and $\bv(u)\neq 0$. Then $H^0(X,\T_X)_u$ equals
\begin{align*}
	H^0\left(\PP^1,\mcL^{-u}\right)\otimes V(u)\chi^{-u}\oplus H^0\left(\PP^1,\CO_{\PP^1}\Big(\gamma(u)+\{\infty\}\Big)\right)\langle e(u),ye(u)-\partial_{\bv(u)} \rangle\chi^{-u}.
\end{align*}
Suppose  still that  	$u(\rho)<1$ for all $\rho\in\xray$ but $\mcP(u)\neq\mcP$ or $\bv(u)=0$.
Then $H^0(X,\T_X)_u$ equals
\begin{align*}
		H^0\left(\PP^1,\mcL^{-u}\right)\otimes N\chi^{-u}\oplus H^0\left(\PP^1,\CO_{\PP^1}\Big(\gamma(u)+2\{\infty\}\Big)\right)e(u)\chi^{-u}.
\end{align*}
Suppose instead that $u(\rho)=1$ for some unique $\rho\in\xray$ with $u(\rho')<1$ for all other $\rho'\in\xray$. Then $H^0(X,\T_X)_u$ equals
\begin{align*}
	H^0\Big(\PP^1,\mcL^{-u}\Big)\cdot\partial_\rho\chi^{-u}.
\end{align*}
Finally, if $u$ does not fulfill any of these conditions, $H^0(X,\T_X)_u=0$. 

\end{prop}
\begin{proof}
	If $\mcP(u)\neq \mcP$ or if $\bv(u)=0$, then the pushforward of the twist by $\CO_Z(u)$ of the exact sequence \eqref{eqn:atiyah} must split, and local calculations show that
\begin{align*}
	\psi_*(\pi_*(\T_W)_u)= \left(\mcL^{-u}\otimes N_\QQ\right)\oplus e(u)\cdot \CO_{\PP^1}\Big(\gamma(u)+2\cdot\{\infty\}
	\Big).
\end{align*}
On the other hand, if $\mcP(u)=\mcP$ and $\bv(u)\neq 0$, then we have
\begin{align*}
	\psi_*(\pi_*(\T_W)_u)= \left(\mcL^{-u}\otimes V(u)\right)\oplus \langle e(u),ye(u)-\partial_{\bv(u)} \rangle \cdot \CO_{\PP^1}\Big(\gamma(u)+\{\infty\}\Big).
\end{align*}
	The claims now follow from Proposition \ref{prop:gs}. \end{proof}

For the general case where $Z$ is not separated, we may cover $Z$ by a finite number of weighted projective lines. The space $H^0(X,\T_X)_u$
may then be computed by using Proposition \ref{prop:gstan} to first calculate sections on each chart, and then by intersecting a finite number of finite-dimensional vector spaces.

\begin{table}
	\begin{center}
\begin{tabular}{cccccc}
\toprule
Name & $(-K_X)^3$ & $b_2$ & $\frac{1}{2}b_3$ & $h^0(\T_X)$ & $h^1(\T_X)$\\
\midrule
Q & $ 54$ & $ 1$ & $ 0$ & $ 10$ & $ 0$\\ 
2.24 & $ 30$ & $ 2$ & $ 0$ & $ 2$ & $ 3$\\ 
2.29 & $ 40$ & $ 2$ & $ 0$ & $ 4$ & $ 0$\\ 
2.30 & $ 46$ & $ 2$ & $ 0$ & $ 7$ & $ 0$\\ 
2.31 & $ 46$ & $ 2$ & $ 0$ & $ 7$ & $ 0$\\ 
2.32 & $ 48$ & $ 2$ & $ 0$ & $ 8$ & $ 0$\\ 
3.8 & $ 24$ & $ 3$ & $ 0$ & $ 2$ & $ 5$\\ 
3.10 & $ 26$ & $ 3$ & $ 0$ & $ 2$ & $ 4$\\ 
3.18 & $ 36$ & $ 3$ & $ 0$ & $ 3$ & $ 0$\\ 
3.19 & $ 38$ & $ 3$ & $ 0$ & $ 4$ & $ 0$\\ 
3.20 & $ 38$ & $ 3$ & $ 0$ & $ 4$ & $ 0$\\ 
3.21 & $ 38$ & $ 3$ & $ 0$ & $ 4$ & $ 0$\\ 
3.22 & $ 40$ & $ 3$ & $ 0$ & $ 5$ & $ 0$\\ 
3.23 & $ 42$ & $ 3$ & $ 0$ & $ 6$ & $ 0$\\ 
3.24 & $ 42$ & $ 3$ & $ 0$ & $ 6$ & $ 0$\\ 
4.4 & $ 32$ & $ 4$ & $ 0$ & $ 2$ & $ 0$\\ 
4.5 & $ 32$ & $ 4$ & $ 0$ & $ 4$ & $ 2$\\ 
4.7 & $ 36$ & $ 4$ & $ 0$ & $ 4$ & $ 0$\\ 
4.8 & $ 38$ & $ 4$ & $ 0$ & $ 5$ & $ 0$\\
\bottomrule\\
\end{tabular}

	\end{center}
	\caption{$h^0(\T_X)$ and $h^1(\T_X)$ for complexity-one Fano threefolds}\label{table:fanotan}
\end{table}

\begin{ex}[Complexity-one Fano Threefolds]\label{ex:fanotan}
A list of smooth non-toric Fano threefolds with $(\CC^*)^2$ action has been given by the second author, see \cite{suess:picbook}.
	We may apply the above results to compute $h^0(\T_X)$ for such varieties, see Table \ref{table:fanotan}. The quantity $h^1(\T_X)$ follows from Riemann-Roch and the Betti numbers of $X$, see e.g. \cite[Proposition 3.1]{ilten:fanodegen}. The names appearing in the leftmost column refer to the terminology of Mori and Mukai \cite{mori:81a}.

\end{ex}

\section{Splitting of equivariant bundles on $\PP^n$}\label{sec:splitting}
Let $m,n\in\NN$ with $1\leq m\leq  n$ and 
consider the $(K^*)^m$-action on $\PP^n$ such that the $i$th coordinate of $(K^*)^m$ acts on the $i$th homogeneous coordinate of $\PP^n$ by multiplication.
\begin{theorem}
	Let $\V$ be a rank $r$ vector  bundle on $\PP^n$, which is $(K^*)^m$-equivariant with respect to the above action. If $r<\min\{n,m+3\}$ then $\V$ splits equivariantly as the sum  of line bundles.
\end{theorem}
\begin{proof}
	If $n=2$, the statement is trivial. We proceed by induction on $r$, assuming that $n>2$. Suppose the statement is true for all $r\leq k$, and let $\V$ be  a 	$(K^*)^m$-equivariant rank $k+1$ bundle on $\PP_K^n$, where $k+1<\min\{n,m+3\}$.
	Let $P_0$ be the closed point of $\PP_K^n$ given by the vanishing of all but the $m$th coordinate, and let $H$ be the divisor on which the $m$th coordinate vanishes. We restrict $(K^*)^m$ to the subtorus $K^*$ corresponding to the $m$th factor. 
Of course, the bundle $\V$ is equivariant with respect to this action.	
	This	
	$K^*$ acts freely on $W=\PP^n\setminus(H\cup P_0)$; the quotient is $\PP^{n-1}$ with the map just given by projection from $P_0$ to the hyperplane $H$. Thus, $\xray$ consists of the single divisor $H$, and rank $k+1$ equivariant vector bundles on $\PP^n$ are given by a filtration $\{\E(i)\}_{i\in\ZZ}$ of a  rank $k+1$ vector bundle $\E$ on $\PP^{n-1}$ satisfying the compatibility conditions of \ref{def:compatible}.

	First, we suppose that there exists $\lambda\in\ZZ$ such that $\E=\E(\lambda)$ and $\E(\lambda+1)=0$. Since the point $P_0$ is in the closure of every orbit of $W$, the compatibility condition implies that $\E$ splits as a sum of invertible sheaves $\bigoplus_{j=1}^{k+1}\mcL_j$. Then the filtrations $\{\mcL_j(i)\}_{i\in\ZZ}$ defined as 
	$$\mcL_j(i)=\begin{cases} 
		\mcL_j & i\leq \lambda\\
		0& i>\lambda
	\end{cases}
$$
for $j=1\ldots k+1$, correspond to $K^*$-equivariant line bundles on $\PP^n$ whose direct sum is the original bundle. Since a vector bundle splits $T$-equivariantly if and only if it splits (see the proof of \cite[1.2.3]{klyachko:89a}), $\V$ must split equivariantly.

Suppose instead that there exists $\lambda\in\ZZ$ such that $\F=\E(\lambda)$ is neither $0$ nor $\E$ and $\E(\lambda+1)=0$.
Consider the exact sequence
\begin{equation}\label{eqn:split}
\begin{CD}
0@>>>\F@>>>\E@>>>\Q@>>>0,
\end{CD}
\end{equation}
where $\Q$ is the quotient $\E/\F$. Note that $\Q$ is locally free, since $\E(i)$ is a filtration by sub-bundles.

We claim that $\F$ and $\Q$ split as sums of line bundles.
If $m>1$, a residual $(K^*)^{m-1}$ acts on $\PP_K^{n-1}$ and $\E$ and $\F$ are equivariant with respect to this action, and hence $\Q$ as well. It follows by the induction hypothesis that $\F$ and $\Q$ split, since their rank is at most $k$.
If instead $m=1$, then the rank of $\V$ is at most $3$. Hence, the rank of $\F$ must be $1$ or $2$, and by replacing $\V$ with $\V^*$, we can assume that $\F$ has rank $1$ and $\rk \Q\leq 2$. As above, the compatibility condition at $P_0$ implies that $\E$ splits. Due to the splitting of $\E$ and $\F$ and the fact that $n>2$, the long exact sequence of cohomology coming from \eqref{eqn:split} gives that $H^1(\PP^{n-1},\Q(t))=0$ for all $t\in\ZZ$. Hence, $\Q$ must also split by the Evans-Griffith criterion \cite[Theorem 2.4]{evans:81a}.

Now, since $n>2$ and $\F$ and $\Q$ are split, we have that 
$$
\Ext^1(\Q,\F)=H^1(\PP^{n-1},\Q^{\vee}\otimes\F)=0
$$
so the above sequence splits and we can view $\Q$ as a direct summand of $\E$. 
The two filtrations
\begin{align*}
\F(i)=\F\cap \E(i) \qquad \Q(i)=\Q\cap \E(i)
\end{align*}
correspond to equivariant bundles of rank at most $k$ on $\PP^n$ whose direct sum is the original bundle. The theorem now follows by induction.
\end{proof}

\begin{rem}
	The hypothesis above that $\rk \V <n$ is necessary, since the tangent bundle does not split. We do not know if the hypothesis $\rk \V<m+3$ is necessary. If one assumes Hartshorne's conjecture (rank two bundles on $\PP^n$ split for $n\geq 7$), then the above proof may be adapted to show that this hypothesis can be relaxed to $\rk \V<m+5$ so long as $n\geq m+7$. 
	
The following question precisely describes the obstruction to omitting the hypothesis on $m$.
\begin{question}
	Is there an exact sequence of vector bundles on $\PP^n$ for some $n\geq 2$ as in Equation \eqref{eqn:split} such that
	\begin{enumerate}
		\item $\F$ and $\Q$ do not split as  sums of line bundles;
		\item $\rk \E\leq n$ and $\E$ splits as a sum of line bundles?
	\end{enumerate}
\end{question}
If the answer to this question is negative, then the bundles appearing in \eqref{eqn:split} in the proof of the theorem must split, regardless of assumptions on $m$. Conversely, if there is such a sequence, one may use it to construct a $K^*$-equivariant bundle on $\PP^{n+1}$ of rank  $\rk \E$ which does not split as a sum of line bundles. 
\end{rem}

\section{Toric Downgrades}\label{sec:downgrades}
Consider a toric variety $X$ with embedded torus $T'$ and a $T'$-equivariant vector bundle $\V$ on $X$. This may be described via a collection of filtrations of a vector space, see Section~\ref{sec:toric-situation}.
If $T \subset T'$ is a subtorus, then $\V$ is also $T$-equivariant. The aim of this section is to translate the description of $\V$ as a $T'$-equivariant bundle into a description of $\V$ as a $T$-equivariant bundle. In other words, we need to give a description of $Z=[W/T]$, $\xray$ and the $(Z,\xray)$-filtration corresponding to $\V$ in terms of (stacky) toric geometry.

\subsection{Toric Stacks}
The first thing we need is a description of non-separated tame toric stacks with finite stabilizers, since they occur as quotients  in our construction. In principle this is done in \cite{borisov2005orbifold}, but we need to enrich their concept of \emph{stacky fans} to cover  the non-smooth and non-separated cases.
The non-smooth case is treated in \cite[Section~4]{tyomkin12}. We will use essentially the same construction here, but  modify it slightly to include non-separated stacks. Before dealing with stacks, we will first discuss toric prevarieties, see also \cite{acampo:prevars} for non-separated toric constructions.

\begin{definition}
\label{sec:def-prefan}
  A \emph{prefan} is a pair $(\Sigma,\mu_\Sigma)$ consisting of a finite poset $(\Sigma,\leq)$ with a unique minimal element $0$ and a morphism $\mu_\Sigma$ of posets associating to each element of $\Sigma$ a cone in $N_\QQ$, such that the following conditions are fulfilled:
  \begin{enumerate}
  \item $\mu_\Sigma(0)=0$,
  \item for $\Sigma_{\leq \sigma}:=\{\tau \mid \tau \leq \sigma\}$, the restriction of $\mu_\Sigma$ is an isomorphism of posets between $\Sigma_{\leq \sigma}$ and the set of faces of $\mu_\Sigma(\sigma)$.
  \end{enumerate} 

We define $\sigma \cap \sigma'$ to be the subposet $\Sigma_{\leq \sigma} \cap \Sigma_{\leq {\sigma'}}$. A \emph{ray} of $\Sigma$ is an element $\rho \in \Sigma$ such that $\mu_\Sigma(\rho)$ is a cone of dimension one. $\Sigma^{(1)}$ denotes the set of all rays.
\end{definition}

 One may view our prefans as certain abstract polyhedral complexes with a geometrical embedding for every single polyhedron. In particular, a cone may occur in several instances in our complexes and all cones are glued only along faces which are subsets of their common intersection. The drawback of this kind of object is that it is not as easy to draw as a usual fan. To overcome this problem, one can draw several pictures of fans, such that the image of every maximal element occurs in at least one of them, and in the pictures, intersection respects the partial ordering. However, one must also identify faces occurring in different pictures. This idea is formalized in \cite{acampo:prevars} under the name \emph{system of fans}.
 
\begin{construction}
	For any cone $C$ in $N_\QQ$, let $\tv(C,N)=\tv(C)$ denote the affine toric variety corresponding to $C$, see e.g. \cite{fulton:93a}.
  For two elements $\sigma,\sigma' \in \Sigma$  the set $\mu_\Sigma(\sigma \cap \sigma')$ defines a subfan of $\mu_\Sigma(\sigma)$ and $\mu_\Sigma(\sigma')$, respectively. Hence, we may identify the corresponding open subvarieties of $\tv({\mu_\Sigma(\sigma))}$
and $\tv({\mu_\Sigma(\sigma')})$. This leads to an equivalence relation $\sim$ satisfying the necessary cocycle condition and we obtain a toric prevariety
\[\tv(\Sigma):=\left(\coprod_{\sigma \in \Sigma} \tv(\mu_\Sigma(\sigma))\right) / \sim.\]
\end{construction}

\begin{rem}
  Every fan $\Delta$ gives rise to a prefan $\Delta^p$
by considering the poset $\Delta$ and the identity map.
On the other hand, if $\Sigma$ is a prefan such that $\mu_\Sigma$ is an isomorphism of posets and $\mu_\Sigma(\Sigma)$ is a fan, then by definition  $\tv(\Sigma)=\tv({\mu_\Sigma(\Sigma)})$. Hence, by abuse of notation we will treat every fan as a special case of a prefan.
\end{rem}

\begin{ex}
\label{sec:exmp-double-point}
   We consider the prefan given by $\Sigma=\{0,1,1',-1\}$ with partial order $0\leq 1,1',-1$, and the mapping $\mu$ with $\mu(0)=0$, $\mu(1)=\mu(1')=\QQ_{\geq 0}$, $\mu(-1)=\QQ_{\leq 0}$. Our construction gives a projective line with a doubled origin. The two instances of this point corresponds to $1$ and $1'$ which are mapped to the same cone.

We may sketch this situation as in Figure~\ref{fig:doublepoint}(a). Alternatively, we can give the two fans in \ref{fig:doublepoint}(b). From the labeling of the cones we can read off the $\mu$. The partial order is given from the face relation in both pictures.

   \begin{figure}[htbp]
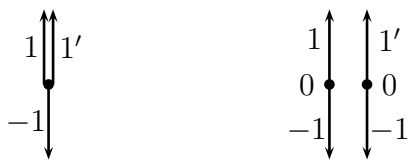

     \centering
     \subfigure[Sketch of a prefan]{\prefanb}
     \subfigure[A system of a fans]{\systemfanb}
    \caption{Projective line with a doubled point}
    \label{fig:doublepoint}
   \end{figure}
 \end{ex}

\begin{definition}[cf. {\cite[Definition 4.1]{tyomkin12}} ]
  A \emph{stacky prefan} is a pair $\boldsymbol\Sigma=(\Sigma,\Sigma^0)$ consisting of a  prefan $\Sigma$ in $N_\QQ$ and a mapping $\Sigma^0$ associating to every $\sigma \in \Sigma$ a sublattice $N^0_\sigma:=\Sigma^0(\sigma) \subset N$ such that
  \begin{enumerate}
  \item $N_\sigma^0$ is a sublattice of finite index in $N_\sigma:=\QQ\cdot \mu_\Sigma(\sigma) \cap N$, 
  \item $\Sigma^0$ respects the order of $\Sigma$, i.e $N_\tau^0 = N_\tau \cap N_\sigma^0$ for $\tau \leq \sigma$.
  \end{enumerate}
\end{definition}

We will now show how to associate a tame toric stack $X(\boldsymbol\Sigma)$ with finite and generically trivial stabilizers  to any stacky prefan $\Sigma$. Our construction is identical to that of \cite[Section 4.1]{tyomkin12}, except that we glue things in a more complicated fashion.

\begin{construction}
\label{sec:const-toric-stack-local}
Locally we may construct the stack corresponding to $\boldsymbol\Sigma$ as the quotient of an affine toric variety by a finite subgroup of the torus. Let us first assume that $\mu_\Sigma(\sigma)$ is full-dimensional.
We have an  exact sequence
\[\xymatrix@1{0  \ar[r] & N_\sigma^0  \ar[r]^\iota & N  \ar[r] & G_\sigma  \ar[r] & 0}\]
where $G_\sigma$ is some finite abelian group.
 By dualizing the embedding we get
\[\xymatrix@1{0  \ar[r] & M  \ar[r] & M_\sigma^0  \ar[r] & G_\sigma  \ar[r] & 0},\]
which leads to an embedding $G_\sigma \cong \Spec K[G_\sigma] \hookrightarrow \Spec K[M^0_\sigma]=T_\sigma^0$. 
We then define $U_\sigma$ to be the quotient stack $[\widetilde{U}_\sigma/G_{\sigma}]$, where $\widetilde{U}_\sigma=\tv(\mu_\Sigma(\sigma),N^0_\sigma)$.

If $\mu_\Sigma(\sigma)$ is not of maximal dimension, then we may replace $N_\sigma^0$ by some sublattice $N' \subset N$ of finite index, such that $N' \cap N_\sigma
= N^0_\sigma$. It is not hard to see that the corresponding quotient stack does not depend on the choice of $N'$ (see e.g. \cite[Theorem B.3]{geraschenko11}).
In particular, for $\sigma \geq \tau \leq \sigma'$ the corresponding open substacks $[\tv(\tau,N^0_\sigma)/G_{\sigma}] \subset U_\sigma$ and $[\tv(\tau,N^0_{\sigma'})/G_{{\sigma'}}] \subset U_{\sigma'}$ can be identified. By gluing along these common open substacks we obtain a global object $\tv(\boldsymbol\Sigma)$.
By construction, $\tv(\boldsymbol\Sigma)$ is a tame stack with finite and generically trivial stabilizers. Moreover, the face $0 \leq \sigma$ gives an embedding $T \hookrightarrow [\tv(\mu_\Sigma(\sigma),N^0_\sigma)/G_{\sigma}]$ with $T = T^0_\sigma/G_\sigma$, and the $T^0_\sigma$-action on $\tv(\mu_\Sigma(\sigma))$ descends to a $T$-action on the quotient. These actions glue and we obtain a $T$-action on $\tv(\boldsymbol\Sigma)$. 

It is not difficult to check that the coarse moduli space of $\tv(\boldsymbol\Sigma)$ is the toric pre-variety $\tv(\Sigma)$; we leave the details to the reader.
\end{construction}

\begin{rem}
Suppose that the stacky fan $\boldsymbol\Sigma$ satisfies the condition that for any $\sigma_1,\sigma_2\in \Sigma$, the intersection $\sigma_1\cap\sigma_2$ has a unique maximal element $\tau$. Then $\tv(\boldsymbol\Sigma)$ is a global quotient of a toric variety by a sub-quasitorus. Its fan is obtained by successively embedding copies of the cones $\mu_\Sigma(\sigma_1)$ and $\mu_\Sigma(\sigma_2)$ in $N^0_\sigma \times N^0_{\sigma'}$, and dividing out by the sublattice $N^0_\tau \times (-N^0_\tau)$; we leave it to the reader to fill in the details.
\end{rem}

\subsection{Vector bundles on toric stacks}\label{sec:stackvb}
The description of equivariant vector bundles on our toric stacks is completely parallel to the case of varieties and fans. Indeed, we may construct our vector bundles locally. Here, giving a $T$-equivariant vector bundle on a local chart $U_\sigma=[\widetilde{U}_\sigma/G_{\sigma}]$  is the same as giving a $T_{N^0_\sigma}$-equivariant vector bundle on $\widetilde{U}_\sigma=\tv(\mu_\Sigma(\sigma))$.  Hence, a vector bundle on the toric stack $\tv(\boldsymbol\Sigma)$ is given by a vector space $E$ and filtration $E^\rho(i)$ for every ray $\rho \in \Sigma^{(1)}$ such that Klyachko's compatibility condition \cite[Theorem 2.2.1]{klyachko:89a} is fulfilled with respect to the lattice $N^0_\sigma$. More precisely, for every $\sigma \in \Sigma$, there must exist a splitting $E=\bigoplus_{u \in M^0_\sigma} E^{[\sigma]}(u)$ such that $E^\rho(i) = \bigoplus_{\langle u, \rho^0 \rangle \geq i} E^{[\sigma]}(u)$ if $\rho$ is a ray of $\sigma$.

We obtain an equivalence of categories between families of compatible filtrations $E^\rho(i)$ indexed by $\Sigma^{(1)}$ and vector bundles on $\tv(\boldsymbol\Sigma)$.

\subsection{Downgrades} 
We now turn to the problem of downgrading the description of toric vector bundles. Consider the toric variety $X=\tv(\Sigma)$, where $\Sigma$ is a fan in $N_\QQ'$.
A torus inclusion $T \hookrightarrow T'$ corresponds to an exact sequence of lattices
\begin{equation}
  \label{eq:downgrade-sequence}
  \xymatrix@1{0  \ar[r] & N  \ar[r]^F & N'  \ar[r]^\mu & \overline N  \ar[r] & 0}
\end{equation}
The subtorus $T\subset T'$ also acts on $X$.
It is easy to see that the DM-locus  $W \subset X$ with respect to  $T$  corresponds to the subfan $\Sigma' \subset \Sigma$ consisting of those cones which preserve their dimension under $\mu_\QQ$. 

Furthermore, $\xray$ is  $\Sigma^{(1)}\cap N_\QQ$, or more precisely, $\xray$ consists of the primitive generators of those rays which are contracted by $\mu$.

\begin{lemma}
\label{sec:lem-stacky-quot}
  Consider the stacky prefan $\boldsymbol\Sigma$  given by the  prefan
$(\Sigma',\mu)$
together with the map $\sigma \mapsto \mu(N_\sigma')$. Then 
$$\tv(\boldsymbol\Sigma)=[W/T].$$
\end{lemma}
\begin{proof}
  Note that since by assumption $\sigma$ is not contracted by $\mu$, the restriction $\mu|_{N_\sigma}$ is an isomorphism onto its image. Locally, we have maps of cones fitting into a commutative diagram
  \begin{equation}
    \label{eq:two-quotients}
    \xymatrixcolsep{5pc}\xymatrix{
  N' \ar[d]^\mu \ar@^{<-)}[r]^{(\mu|_{N_\sigma})^{-1}}  \POS p+(0,6) *+{\sigma}="s"   & \mu(N_\sigma) \ar@^{(->}[d]^\iota \POS p+(0,6) *+{\mu(\sigma)} \ar@{->} "s"\\
  \overline N \ar@{=}[r] & \overline N     
}
  \end{equation}
By \cite[Theorem B.3]{geraschenko11} this induces an isomorphism 
$$[\tv(\sigma)/T]\cong \left[\tv\big(\mu(\sigma),\mu(N'_\sigma)\big)/G_{\sigma}\right].$$ Moreover, these local isomorphisms glue and we obtain the statement of the lemma.
\end{proof}

Now, consider a $T'$-equivariant (i.e. toric) vector bundle $\V$ on $X$. This corresponds to filtrations $E^\rho(i)$ of a vector space $E$, where $\rho$ ranges over all $\rho\in\Sigma^{(1)}$.
Restricting our torus action, we may instead consider $\V$ as a $T$-equivariant bundle. This corresponds to some compatible $(Z,\xray)$-filtration $(\E,\{\E^\rho(i)\})$, where $Z=\tv(\boldsymbol\Sigma)$.
Since the quotient torus $\overline T=T'/T$ acts on $(\E,\{\E^\rho(i)\})$, we may describe this as in Section \ref{sec:stackvb}.

\begin{prop}
	The $\overline T$-equivariant vector bundle $\E$ on $\tv(\boldsymbol\Sigma)$ corresponds to the filtrations $E^\rho(i)$ for every ray in $\Sigma'$. Moreover, for $\gamma \in \xray \subset N$ and $j\in\ZZ$, the bundle $\E^{\gamma}(j)$ corresponds to the filtrations $E^{\rho}(i)\cap E^{\gamma}(j)$.
\end{prop}
\begin{proof}
The first statement is local with respect to cones in $\Sigma'$, so without loss of generality we may assume that $\Sigma'$ consists of a single cone $\sigma$.
The proof of Lemma~\ref{sec:lem-stacky-quot} implies that we have the following commutative diagram of quotients:
 \[ \xymatrixcolsep{5pc}\xymatrix{
  \tv(\sigma) \ar[d]^{\pi=\cdot/T}\ar@^{<-)}[r]^i    &  \widetilde{U}_\sigma \ar[d]^{\cdot/G_\sigma} \\
  Z \ar@{=}[r] & Z }\]
Now, we have a bijection between $\overline T$-equivariant vector bundles on $Z$ and $T^0_\sigma$ -equivariant vector bundles on $\widetilde{U}_\sigma$, realized by pullback and pushforward of the $G_\sigma$-invariant part, respectively. Moreover, we have a bijection between $\overline T$-equivariant vector bundles on $Z$ and $T'$-equivariant vector bundles on $\tv(\sigma)$. This correspondence is realized by pullback and pushforward of the $T$-invariant part, respectively.

We would like to show that $\E=\pi_*(\V)^T$ corresponds to the chosen filtrations. By our construction, it is equivalent to show that  
$i^*(\V)$ corresponds to the given filtrations, but this is well known for toric vector bundles.

Now we have to check that $\E^{\gamma}(j)$ actually corresponds to the filtrations $E^\rho(i) \cap E^\gamma(j)$. By the same argument as in the proof of Lemma~\ref{lemma:stackfil1}, we may reduce to the situation that $\V$ has rank one and is generated by a $T'$-semi-invariant section of weight $v$. The section is also semi-invariant with respect to the $T$-action. The corresponding weight is given as $F^*(v)$, where $F^*:M' \rightarrow M$ is the dual homomorphism to $F$ in (\ref{eq:downgrade-sequence}). 

Now, we only have to show that the rank of $\E^\gamma(j)$ jumps (from $0$ to $1$) at the same index as the dimension of $E^{\gamma}(j)$. To prove this we consider $X$ as both a $T'$-variety and as a $T$-variety. In both cases we have $\gamma \in \xray_{T'}$ and $F(\gamma) \in \xray_{T}$, respectively. Now, $$v(F(\gamma))=F^*(v)(\gamma)$$
holds and the proof of Lemma~\ref{lemma:stackfil1} shows that the jumping position of $E^{\gamma}(j)$ is calculated by the left-hand side and that of $\E^{\gamma}(j)$ by the right-hand side.
\end{proof}

\begin{ex}
\label{sec:example-downgrades-blowup}
We consider $X$ to be the blowup of $\PP^2$ in one point and the $K^*$-action given by multiplication on the first coordinate. This situation corresponds to the fan in $N'_\QQ=\QQ^2$ with rays  $\rho_1=\QQ_{\geq 0}\cdot e_1$, $\rho_2=\QQ_{\geq 0}\cdot (e_1+e_2)$, $\rho_3=\QQ_{\geq 0}\cdot e_2$, $\rho_4=\QQ_{\geq 0}\cdot (-e_1-e_2)$ along with the exact sequence
\begin{equation}
  \label{eq:downgrade-sequence1}
   \xymatrixcolsep{4pc}\xymatrix@1{0  \ar[r] & \ZZ \ar[r]^{F={1\choose 0}} & N'=\ZZ^2  \ar[r]^{\mu=(0,1)} & \overline N =\ZZ \ar[r] & 0}
\end{equation}
The situation is sketched in Figure~\ref{fig:project}(a).
  \begin{figure}[htbp]
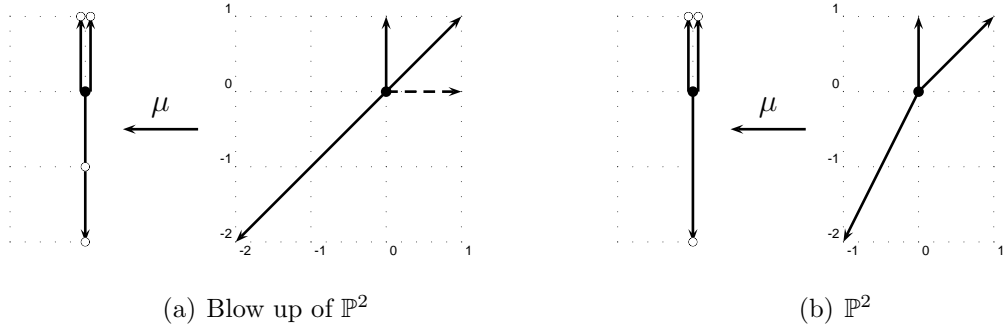

    \centering
    \subfigure[Blow up of $\PP^2$]{\fanbproject} \hspace{1cm}
    \subfigure[$\PP^2$]{\fanfproject}
    \caption{Two toric quotient maps}
    \label{fig:blowupquotient}
\label{fig:project}
  \end{figure}
Now, the fan $\Sigma'$ of $W$ consists only of the rays $\rho_2$, $\rho_3$ and $\rho_4$, respectively. Since $e_1$ lies in the kernel of $\mu$, the corresponding ray $\rho_1$ is contracted by $\mu$ and has to be removed, as well as the full-dimensional cones. Now, the corresponding prefan is the same as in Example~\ref{sec:exmp-double-point}. The stacky structure is trivial, i.e. $\Sigma^0$ always gives the maximal lattice (in the figure the semi-group $\Sigma^0(\sigma) \cap \mu(\sigma)$ is sketched by the $\circ$-symbols).

By \cite[2.3 Example 5]{klyachko:89a} or Theorem \ref{thm:tangent}, the tangent bundle on $X$ is given by the filtrations
\begin{equation}
  \label{eq:tangentbundle}
  E^\rho(i)=\begin{cases}
N'\otimes k & i \leq 0\\
\langle \rho \rangle & i =1\\
0 & i>1
\end{cases}.
\end{equation}

By the downgrading procedure $\E$ is given by these filtrations for the rays $\rho_2$, $\rho_3$, $\rho_4$. The set $\xray$ consists of $\rho_1$, and the corresponding filtration of $\E$ has one non-trivial step $\E^{\rho_1}(1)$ which is a rank one sub-bundle $\F \subset \E$ given by the filtrations $F^{\rho_j}(i)$ for $j=2,3,4$ with
\begin{align*}
  F^{\rho_j}(i)&=\begin{cases}
\langle \rho_1 \rangle & i \leq 0\\
0 & i > 0 
\end{cases}\\
\end{align*}
\end{ex}
It is not hard to see that 
\begin{equation}
  \label{eq:blowup-esplitting}
  \E \cong \CO([0_1]+[\infty]) \oplus \CO([0_2])
\end{equation}
 and $\F \cong \CO_Z$. Here $0_1,0_2$ denote the two instances of the origin. The inclusion $\CO_Z\cong \F \hookrightarrow \E$ is given by sending $1$ to $e_1-e_2$, where $e_1$ and $e_2$ are sections with value $1$ from the line bundles in the decomposition (\ref{eq:blowup-esplitting}). 
 This agrees with the description of $\T_X$ coming from Theorem \ref{thm:tangent}.
 \begin{ex}

Now consider $\PP^2$ with a $K^*$-action given by the weights $1,-1,0$. The toric variety $\PP^2$ is given by the fan $\Sigma$ in $N'_\QQ=\QQ^2$ spanned by the rays $\rho_1= \QQ_{\geq 0} \cdot(e_1+e_2)$, $\rho_2= \QQ_{\geq 0} \cdot(e_2)$ and $\rho_3= \QQ_{\geq 0} \cdot(-e_1-2e_2)$. With this fan the $K^*$ action corresponds also to the exact sequence~(\ref{eq:downgrade-sequence1}). Now, $\Sigma'$ consists of all non-maximal cones, since no ray is contracted by $\mu$. The quotient prefan is the same as in the previous example, but the stacky structure differs: $\Sigma^0(\rho_3)=2\overline N$. The situation is sketched in Figure~\ref{fig:project}(b).

The tangent bundle is again given by the filtrations $E^{\rho_j}(i)$ from (\ref{eq:tangentbundle}) for $j=1,2,3$. These filtrations also define the bundle $\E$ on $Z$. 
For this example $\xray$ is empty and we have no filtration to consider.
\end{ex}

\section{Equivariant Deformations}\label{sec:def}
\subsection{Families of Vector Bundles}
Let $X$ be a $T$-variety and $\V$ a $T$-equivariant vector bundle. Let $S$ be a scheme. We would like to construct $T$-equivariant deformations of $\V$ over $S$, i.e. a $T$-equivariant vector bundle $\tV$ on $X\times S$ restricting to $\V$.

Let $W$ be the DM-locus of $X$, $Z=[W/T]$. Now, any equivariant vector bundle on $X\times S$ corresponds to some $(Z\times S,\xray)$-filtration by Theorem \ref{thm:F}. As one might suspect, such a vector bundle restricts to $\V$ if and only if the corresponding $(Z\times S,\xray)$-filtration restricts to the $(Z,\xray)$-filtration for $\V$.

\begin{prop}\label{prop:family}
	Let $\tV$ be an equivariant vector bundle on $X \times S$ with corresponding $(Z\times S,\xray)$-filtration $(\tE,\{\tE^\rho(i)\})$. Then for any point $t \in S$ with fiber $X_t\cong X$ in $X\times S$, $\tV|_{X_t}$ corresponds to the $(Z,\xray)$-filtration
	$(\tE|_{Z_t},\{\tE|_{Z_t}^\rho(i)\})$.
\end{prop}
\begin{proof} 
We have the following commutative diagram
$$\xymatrix{
W_t \ar[d]^\pi \ar@{^{(}->}[r]^\iota &W \times S\ar[d]^{\pi \times \id}\\
Z_t \ar@{^{(}->}[r]^j & Z \times S 
}.$$
Hence, we obtain 
\begin{align*}
  \tV{|_{W_t}}  &= \iota^*\tV_{|_{W \times X}}\\
                          &= \iota^*(\pi \times S)^*\tE = \pi^*j^*\tE\\
                          &=\pi^*\tE_{|_{Z_t}}
\end{align*}
It remains to show that the filtrations $\{\tE|_{Z_t}^\rho(i)\}$ actually correspond to $\tV{|_{X_t}}$. 

As in the proof of Lemma~\ref{lemma:stackfil1} we reduce to the case that $\tV$ is an equivariant line bundle with sections locally generated by elements $s$ of weight $v$. Now, the filtration consists only of one step jumping from $0=\tE^\rho(i+1)$ to $\tE=\tE^\rho(i)$. Hence, the position $i$ of this step determines the filtration. We may reduce to the local situation where $\tV$ is generated by a single section $s$ of weight $v \in M$, now the jumping position is given by $-v(\rho)$ (see Remark \ref{sec:rem-rank-one-filt}). But when restricting to $X_t$ the weight of $s$ does not change. Hence, the jumping position for the filtration corresponding to $\tV{|_{X_t}}$ is $-v(\rho)$ as well and $\tE|_{Z_t}^\rho(i)$ is indeed the correct filtration for $\tV{|_{X_t}}$.
\end{proof}
\begin{ex}[First Order Deformations]
	Let $X$ be a $T$-variety and $\V$ an equivariant bundle on $X$. Let $S$ be the fat point $\spec K[t]/t^2$. Then it is well-known that equivalence classes of families $\tV$ on $X\times S$ together with a map $\tV\to \V$ restricting to an isomorphism correspond to extension classes in $\Ext_{\CO_X}^1(\V,\V)$, see e.g. \cite[Theorem 2.7]{hartshorne:10a}. Indeed, to an extension
$$	\begin{CD}
0 @>>> \V @>\phi_1>> \tV @>\phi_2>> \V @>>>0
	\end{CD}
$$
one associates the map  $\phi_2:\tV\to \V$ on $X\times S$. Here, the $\CO_{X\times S}$-module structure on $\tV$ is given by $t\cdot s=(\phi_1\circ \phi_2) (s)$. In particular, \emph{equivariant} families $\tV\to \V$ correspond to equivariant extensions, that is, the degree zero part of $\Ext_{\CO_X}^1(\V,\V)$.

Now, let $\{\E^\rho(i)\}$ be the $(Z,\xray)$-filtration corresponding to $\V$. Let $$\Ext^1_{(Z,\xray)}(\{\E^\rho(i)\},\{\E^\rho(i)\})$$ denote the set of classes of extensions of 
the $(Z,\xray)$-filtration $\{\E^\rho(i)\}$ by itself to an $X$-compatible $(Z,\xray)$-filtration. Then by the above proposition, we have the following commutative diagram:
$$
\begin{CD}
	\Ext_{\CO_X}^1(\V,\V)_0 @>>> \{\phi:\tV\to \V\ |\ \phi|_{t=0}\ \textrm{is an isomorphism}\}/\sim\\ 
	@V\bF VV @V\bF VV\\
	\Ext_{(Z,\xray)}^1\left(\{\E^\rho(i)\},\{\E^\rho(i)\}\right) @>>> \left\{\psi:\widetilde{\E}^\rho(i)\to \E^\rho(i)\ \Big|\ \begin{array}{c}\psi|_{t=0}\ \textrm{is an isomorphism}\\
		\widetilde{\E}^\rho(i)\ 		\textrm{is $X$-compatible}
	\end{array}	
	\right\}/\sim\\ 
.\end{CD}
$$
The horizontal arrows are bijections, and if $X$ satisfies \eqref{dagger}, then so are the vertical arrows. Indeed, by Theorem \ref{thm:compatible}, the left vertical arrow is bijective, from which it follows that the right arrow must be as well.
Hence, we may study so-called first order deformations of $\V$ in terms of objects on the quotient $Z$.
\end{ex}

\subsection{Deformations of the Tangent Bundle}
One general problem  in deformation theory is to lift a given infinitesimal deformation to higher order. However, even for toric vector bundles, there are often obstructions to lifting deformations. In fact, moduli spaces of toric vector bundles can have arbitrarily bad singularities, see \cite[Section 4]{payne:08a}.
In the following, we will restrict our attention to deformations of the tangent bundle; these play an essential role in $(0,2)$ mirror symmetry, see e.g. \cite{melnikov:11a}.
For the case of toric Fano varieties, we may easily calculate obstructions to lifting:

\begin{theorem}\label{thm:obstructions}
  Let $X$ be a smooth, complete toric Fano variety corresponding to the fan $\Sigma$. For each ray $\rho\in \Sigma^{(1)}$, let $D_\rho$ denote the associate invariant prime divisor. 
Then for $i\geq 2$,
$$\Ext^i(\T_X,\T_X)\cong \bigoplus_{\gamma\neq\rho\in\Sigma^{(1)}}H^{i-1}(D_\rho,\CO_{D_\rho}(D_\gamma)).$$ In particular, if each divisor $D_\rho$ is Fano, then the versal deformation of the bundle $\T_X$ is smooth.
\end{theorem}
\begin{proof}
  We will make use of the generalized Euler sequence \cite{euler}:
  \begin{equation}\label{eqn:euler}
    \begin{CD}
      0@>>> \Omega_X @>>> \bigoplus_\rho\CO(-D_\rho) @>>> \CO_X\otimes \Pic(X) @>>>0.
    \end{CD}
 \end{equation}
Twisting by $\T_X$ and using the long exact sequence of cohomology, we can conclude that for $i\geq 2$,
$$
\Ext^i(\T_X,\T_X)\cong H^i(X,\T_X\otimes\Omega_X)\cong H^i(X,\T_X\otimes  \bigoplus_\rho\CO(-D_\rho)). 
$$
Indeed, this follows from the fact that $H^i(X,\T_X)=0$ for $i>0$, see \cite[Proposition 4.2]{brion:96a}.

Now, for any divisor $D_\rho$, we consider the short exact sequence
\begin{equation}\label{eqn:restr} 
   \begin{CD}
      0@>>> \CO(-D_\rho) @>>> \CO_X @>>> \CO_{D_\rho} @>>>0.
    \end{CD}
 \end{equation}
 Since $D_\rho$ is a complete toric variety (cf. \cite[Section 3.1]{fulton:93a}) and higher cohomology groups of the structure sheaf vanish on a toric variety (cf. \cite[Section 3.5]{fulton:93a}), the long exact sequence of cohomology implies that $H^i(X,\CO(-D_\rho))=0$ for all $i$.
Twisting the sequence dual to \eqref{eqn:euler} by $\bigoplus_\rho\CO(-D_\rho)$ and using this vanishing gives
$$
H^i(X,\T_X\otimes  \bigoplus_\rho\CO(-D_\rho))\cong H^i(X,\bigoplus_{\gamma,\rho}\CO(D_\gamma-D_\rho)).
$$

 On the other hand, we may consider the twist of the sequence \eqref{eqn:restr} by $\CO(D_\gamma)$ for any divisor $D_\gamma\neq D_\rho$.
 Since $X$ is Fano, the higher cohomology groups of $\CO_X(D_\gamma)$ vanish, giving an isomorphism
 \begin{equation}\label{eqn:rest2}
 H^i(X,\CO(D_\gamma-D_\rho))\cong H^{i-1}(D_\rho,\CO_{D_\rho}(D_\gamma))
 \end{equation}
 for $i\geq 2$.
The first claim now follows.

The second claim follows from the vanishing of higher cohomology for prime invariant divisors on toric Fano varieties.
\end{proof}

\begin{rem}
  The graded pieces of the cohomology groups appearing in the above theorem may be computed by counting connected components of certain graphs, see \cite[Proposition 1.1]{ilten:surfacedef}.
\end{rem}

Now, if we restrict our attention to a toric Fano surface $X$, the tangent bundle $\T_X$ has unobstructed deformations. Indeed, each invariant prime divisor is isomorphic to $\mathbb{P}^1$ and hence Fano. Thus, the versal deformation of $\T_X$ is smooth, and it is an interesting challenge to describe a versal family of bundles. 

There are exactly four smooth toric Fano surfaces with non-rigid tangent bundle: $\PP^1\times \PP^1$, and the blowup of $\PP^2$ in $n=1,2,3$ points, which we denote by $\Bl_n\PP^2$.
The corresponding fans are depicted in Figure \ref{fig:fanosurface}.
For each surface, one may compute the graded pieces of the space $\Ext^1(\T_X,\T_X)$ of first order deformations of the tangent bundle using \cite[4.6]{klyachko:89a}, assisted by the software \cite{ilten:toricvb}.
The result of this computation is depicted in Figure \ref{fig:fanosurfacedeg}.
Each dot corresponds to a one-dimensional homogeneous piece of $\Ext^1(\T_X,\T_X)$ in the given weight.

Our description of equivariant bundles may be used to construct  explicit $K^*$-equivariant families of vector bundles whose restrictions to a fat point span the graded pieces
of the vector space $\Ext^1(\T_X,\T_X)$. This provides a sort of skeleton of the versal deformation of $\T_X$.
In the remainder of the section, we do this explicitly for $\PP^1\times \PP^1$ and the blowup of $\PP^2$ in one point. Similar constructions may be made for $\Bl_2(\PP^2)$ and $\Bl_3(\PP^2)$.

\begin{figure}
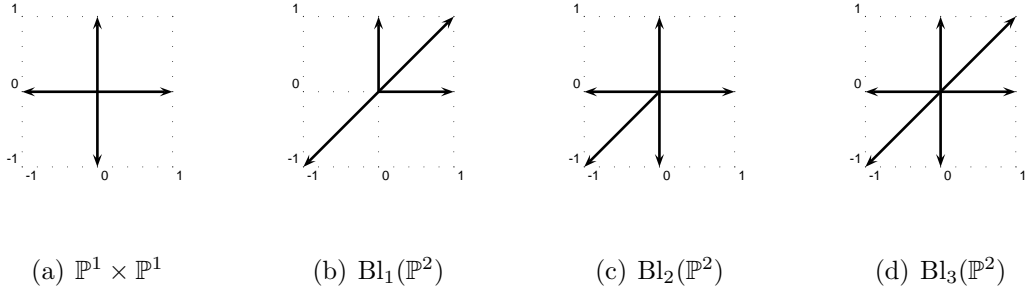

	\begin{center}
		\subfigure[$\PP^1\times\PP^1$]{\fana}
		\subfigure[$\Bl_1(\PP^2)$]{\fanb}
		\subfigure[$\Bl_2(\PP^2)$]{\fanc}
		\subfigure[$\Bl_3(\PP^2)$]{\fand}
		\caption{Fans smooth toric Fano surfaces}\label{fig:fanosurface}
\end{center}
\end{figure}
\begin{figure}
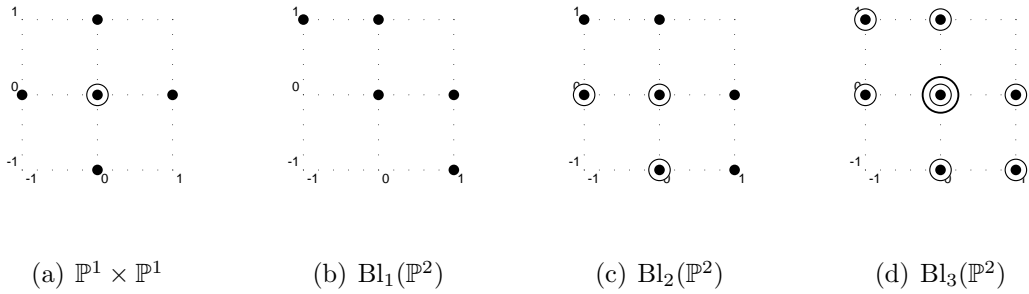

	\begin{center}
		\subfigure[$\PP^1\times\PP^1$]{\dega}
		\subfigure[$\Bl_1(\PP^2)$]{\degb}
		\subfigure[$\Bl_2(\PP^2)$]{\degc}
		\subfigure[$\Bl_3(\PP^2)$]{\degd}
		\caption{Degrees of $\Ext^1(\T_X,\T_X)$ for smooth toric Fano surfaces}\label{fig:fanosurfacedeg}
\end{center}
\end{figure}

\begin{ex}[$X=\PP^1\times\PP^1$]\label{ex:p1p1}
We will consider the subtori $T\subset (K^*)^2$ given as $T=\ker u$ for the characters $u=(1,0)$ and $u=(0,1)$. For both choices of $u$, the $T$-invariant part of $\Ext^1(\T_X,\T_X)$ has dimension $4$, and together they span all of $\Ext^1(\T_X,\T_X)$. Due to symmetry, we will focus on $u=(1,0)$.

With notation as in Section \ref{sec:gen}, $W=\PP^1\times K^*$ and $Z=[W/T]=\PP^1=\Proj K[y_0,y_1]$. The set $\xray$ consists of two elements, $\rho_1=1$ and $\rho_2=-1$. The tangent bundle $\T_X$ corresponds to the $(Z,\xray)$-filtration $(\E,\{\E^\rho(i)\})$, where 
$\E=\CO\oplus\CO(2)$ and 
$$
\E^\rho(i)=\begin{cases}
\CO\oplus\CO(2) & i \leq 0\\
\CO & i =1\\
0 & i>1
\end{cases}.
$$

Let $S=\spec K[t_0,t_1,t_2,t_3,(1-t_0^2)^{-1}]$. Consider the rank two vector bundle $\widetilde{\E}$ on $Z\times S$ with generators $e_1,e_2$ on $y_1\neq 0$ and $f_1,f_2$ on $y_0\neq 0$ and transition functions given by
\begin{align*}
	e_1=f_1+t_0y^{-1}f_2 \qquad e_2=t_0y^{-1}f_1+y^{-2}f_2,
\end{align*}
where $y=y_0/y_1$. 
This bundle is a versal family for $\E$:
for $t_0=0$ it restricts to $\E$, and for $t_0\neq 0$ it restricts to $\CO(1)\oplus \CO(1)$.

We now define two filtrations of $\widetilde\E$: $\widetilde{\E}^\rho(i)=0$ if $i>1$, $\widetilde{\E}^\rho(i)=\widetilde{\E}$ if $i\leq 0$, and
\begin{align*}
	\widetilde{\E}^{\rho_1}(1)=\CO_{\PP^1\times S}\cdot (e_1+t_1e_2+t_2ye_2+t_3f_2)\qquad
	\widetilde{\E}^{\rho_2}(1)=\CO_{\PP^1\times S}\cdot e_1.
\end{align*}
These filtrations are compatible for $X\times S$, so we get a bundle $\widetilde\V$ on this space which restricts to $\T_X$ when $t_i=0$ for all $i$. These four deformation directions span the $T$-invariant part of 
$\Ext^1(\T_X,\T_X)$.

The bundle $\T_X$ is not simple (in fact it is split): $\Ext^0(\T_X,\T_X)$ has dimension two and is concentrated in degree zero. However, any deformation of $\T_X$ as above will be simple, so the number of moduli of the deformed bundle will drop to five.
\end{ex}

\begin{ex}[Blowup of $\PP^2$ in one point]
  Similar to Example \ref{ex:p1p1}, due to symmetry we only need to consider the subtori $T=\ker u$ for the characters $u=(1,0)$ and $u=(1,-1)$. For $u=(1,-1)$, 
$W$ is a nontrivial $K^*$-principal bundle over $Z=\PP^1=\Proj K[y_0,y_1]$ and $\xray$ consists of two elements, $\rho_1=1$ and $\rho_2=-1$. 
Let $S=\spec K[t_1,t_2,t_3]$. Consider the rank two vector bundle $\widetilde{\E}$ on $Z\times S$ with generators $e_1,e_2$ on $y_1\neq 0$ and $f_1,f_2$ on $y_0\neq 0$ and transition functions 
\begin{align*}
	e_1=y^{-1}f_1 \qquad e_2=y^{-1}f_2.
\end{align*}
This is just $\CO(1)\oplus \CO(1)$ extended horizontally along $S$. Now consider the following two filtrations of $\widetilde\E$: $\widetilde{\E}^\rho(i)=0$ if $i>1$, $\widetilde{\E}^\rho(i)=\widetilde{\E}$ if $i\leq 0$, and
\begin{align*}
	\widetilde{\E}^{\rho_1}(1)=\CO_{\PP^1\times S}\cdot (ye_1-e_2+t_1e_1+t_2e_2+t_3ye_2)\qquad
	\widetilde{\E}^{\rho_2}(1)=\CO_{\PP^1\times S}\cdot (ye_1-e_2).
\end{align*}
These filtrations are compatible for $X\times S$, so we get a bundle $\widetilde\V$ on this space which restricts to $\T_X$ when $t_i=0$ for all $i$. These three deformation directions span the $T$-invariant part of 
$\Ext^1(\T_X,\T_X)$.

For $u=(1,0)$, we are in the situation of Example~\ref{sec:example-downgrades-blowup} and we 
get that $Z$ consists of two copies $U_1$ and $U_2$ of $\PP^1$ glued together everywhere except at the origin, and the set $\xray$ consists of a single element $\rho$.
The bundle $\T_X$ corresponds to a $(Z,\xray)$-filtration $(\E,\{\E^\rho(i)\})$. Here, $\E$ is the bundle generated by $e_1^i,e_2^i$ on $U_i\cap[y_1\neq 0]$ and $f_1,f_2$ on $y_0\neq 0$ with transition functions given by
\begin{align*}
 &e_1^1=e_1^2=f_1\qquad &&e_2^1=y^{-2}f_2\\
 &          &&e_2^2=-y^{-1}f_1+y^{-2}f_2
\end{align*}
The filtration $\E^\rho(i)$ is given by 
$$
\E^\rho(i)=\begin{cases}
\E & i \leq 0\\
\CO_Z\cdot f_1 & i =1\\
0 & i>1
\end{cases}.
$$
Now let $S=\spec K[s,t,(s-1)^{-1}]$. Consider the rank two vector bundle $\widetilde{\E}$ on $Z\times S$ with generators $\widetilde{e}_1^i,\widetilde{e}_2^i$ on $U_i\cap[z_1\neq 0]$ and $\widetilde{f}_1,\widetilde{f}_2$ on $y_0\neq 0$ and transition functions 
\begin{align*}
  \widetilde{e}_1^1&=\widetilde{f}_1\qquad &\widetilde{e}_1^2&=\widetilde{f}_1\\
  \widetilde{e}_2^1&=y^{-2}\widetilde{f}_2-2(sy^{-1}+ty^{-2})\widetilde{f}_1\qquad &\widetilde{e}_2^2&=-y^{-1}\widetilde{f}_1+y^{-2}\widetilde{f}_2.
 \end{align*}
Define a filtration $$
\widetilde{\E}^\rho(i)=\begin{cases}
  \widetilde{\E} & i \leq 0\\
  \CO_{Z\times S}\cdot \widetilde{f}_1 & i =1\\
0 & i>1
\end{cases}.
$$
This filtration is compatible for $X\times S$, so we get a bundle $\widetilde\V$ on this space which restricts to $\T_X$ when $s=t=0$. These two deformation directions span the $T$-invariant part of 
$\Ext^1(\T_X,\T_X)$. The $s$-parameter gives a deformation of $\T_X$ in degree $(0,0)$ and the $t$-parameter gives a deformation of $\T_X$ in degree $(1,0)$.
\end{ex}

\bibliographystyle{alpha}
\bibliography{t-vb}

\end{document}